\documentclass[12pt]{amsart}
\usepackage{amssymb,amsmath,amsthm}
\usepackage{mathrsfs}  
\calclayout
\usepackage{graphicx}
\usepackage[all,cmtip]{xy}
\usepackage{tikz-cd}
\usepackage{adjustbox}
\newcommand\scalemath[2]{\scalebox{#1}{\mbox{\ensuremath{\displaystyle #2}}}}
\usepackage{hyperref}
\usepackage[msc-links, backrefs]{amsrefs}
\hypersetup{
  colorlinks   = true,	
  urlcolor     = blue,	
  linkcolor    = blue,	
  citecolor   = red	
}
\newtheorem{theorem}{Theorem}[section]
\newtheorem{proposition}[theorem]{Proposition}
\newtheorem{corollary}[theorem]{Corollary}
\newtheorem{lemma}[theorem]{Lemma}

\newtheorem{remark}[theorem]{Remark}
\numberwithin{theorem}{section} \numberwithin{equation}{section}
\newcommand{\Beq}{\begin{equation}}
\newcommand{\Eeq}{\end{equation}}
\newcommand{\Beqn}{\begin{equation*}}
\newcommand{\Eeqn}{\end{equation*}}
\newcommand{\beq}{\begin{small} \begin{equation}}
\newcommand{\eeq}{\end{equation} \end{small}}
\newcommand{\beqn}{\begin{small} \begin{equation*}}
\newcommand{\eeqn}{\end{equation*} \end{small}}
\begin{document}
\title[Prym varieties for special bielliptic curves]{Geometry of Prym varieties for certain bielliptic curves of genus three and five}
\author{Adrian Clingher}
\address{Dept.\!~of Mathematics \& Statistics, University of Missouri - St. Louis, MO 63121}
\email{clinghera@umsl.edu}
\thanks{A.C. acknowledges support from a UMSL Mid-Career Research Grant.}
\author{Andreas Malmendier}
\address{Dept.\!~of Mathematics, University of Connecticut, Storrs, Connecticut 06269}
\email{andreas.malmendier@uconn.edu}
\thanks{A.M. acknowledges support from the Simons Foundation through grant no.~202367.}
\author{Tony Shaska}
\address{Dept.~of Mathematics \& Statistics,  Oakland University, Rochester, MI 48309}
\email{shaska@oakland.edu}
\begin{abstract}
We construct two pencils of bielliptic curves of genus three and genus five. The first pencil is associated with a general abelian surface with a polarization of type $(1,2)$. The second pencil is related to the first by an unramified double cover, the Prym variety of which is canonically isomorphic to the Jacobian of a very general curve of genus two. Our results are obtained by analyzing suitable elliptic fibrations on the associated Kummer surfaces and rational double covers among them.
\end{abstract}
\maketitle

\section{Introduction and statement of results}
\label{ssec:thms}
Computing isogenies between Jacobian and Prym varieties for curves of genus two and three is considered one of the fundamental problems in the context of computer algebra and encryption as it is closely related to the arithmetic and the discrete logarithm problem in class groups of such curves and Recillas’  \emph{trigonal construction} \cites{MR2406115, MR1736231, MR3389883, MR4063320}.  If the curve of genus three is non-hyperelliptic, there has been no general formula relating its moduli to the moduli of a curve of genus two. In this article, we will derive explicit normal forms for the pencil of plane, bielliptic curves of genus three (and their unramified double coverings by canonical curves of genus five) such that the Prym variety of its general member is 2-isogenous to the Jacobian of a very general curve of genus two. We emphasize that our results are valid for any curve in the moduli space of curves of genus two, not only for special elements or subfamilies.

Let  $\mathcal{C}$ be a smooth curve of genus two defined over the field of complex numbers. Consider a G\"opel subgroup $G' \leqslant \operatorname{Jac}(\mathcal{C})[2]$, i.e., a subgroup maximally isotropic under the Weil pairing.  It is then well known that the quotient $\operatorname{Jac}(\mathcal{C})/G'$ is canonically isomorphic to the Jacobian of a second curve of genus two $\mathcal{C}'$, said to be $(2,2)$-isogenous with $\mathcal{C}$. Moreover, the image of $\operatorname{Jac}(\mathcal{C})[2]$ under the projection map $\Psi' \colon  \operatorname{Jac}(\mathcal{C}) \rightarrow \operatorname{Jac}(\mathcal{C})/G' $ is a G\"opel subgroup of $\operatorname{Jac}(\mathcal{C}')[2]$ and, as $\operatorname{Jac}(\mathcal{C})/ \operatorname{Jac}(\mathcal{C})[2] \simeq  \operatorname{Jac}(\mathcal{C})$, 
one obtains a pair of dual $(2,2)$-isogenies: 
\beq
\xymatrix 
{
 \operatorname{Jac}(\mathcal{C}) 
\ar @/_0.5pc/ @{->} _{\Psi'} [rr] &
&  \operatorname{Jac}(\mathcal{C}')  
\ar @/_0.5pc/ @{->} _{\Psi} [ll] \\
} 
\eeq
The relation between the curves $\mathcal{C}$ and $\mathcal{C'}$ can be made explicit via the Richelot construction \cites{MR1578134, MR1578135}.
\par Consider $G' \simeq \langle \mathscr{L} \rangle \oplus  \langle \mathscr{L}' \rangle$ a marking of the G{\"o}pel subgroup above, with $\mathscr{L}$, $ \mathscr{L}'$ line bundles of order two on the curve $\mathcal{C}$. The line bundle $\mathscr{L}$ determines a canonical {\'e}tale double cover $p \colon \mathcal{H} \rightarrow \mathcal{C}$, with the total space $\mathcal{H}$ being a smooth curve of genus three carrying a base-point free involution $\imath \colon \mathcal{H} \rightarrow \mathcal{H}$. The hyperelliptic involution of $\mathcal{C}$ lifts to a second involution $\jmath \colon \mathcal{H} \rightarrow \mathcal{H}$ that has a fixed locus given by four points. In turn, the involution $\jmath$ defines a canonical bielliptic structure on $\mathcal{H}$, with double cover $\pi \colon \mathcal{H} \rightarrow {\mathcal E}$ mapping to an elliptic curve. The two involutions $\imath$ and $\jmath$ commute, with their composition $\imath \circ \jmath$ defining a hyperelliptic structure on $\mathcal{H}$.  
Also, $\operatorname{Prym}( \mathcal{H}, \pi \colon  \mathcal{H} \rightarrow \mathcal{E})$ is an abelian surface with the curve of genus three $\mathcal{H}$ canonically embedded as a $(1,2)$-polarization \cites{MR0379510, MR572974, MR946234}.
\par Moving up one level, the pull-back $p^* \mathscr{L}'$ is a line-bundle of order-two on the curve $\mathcal{H}$. As such, it defines an {\'e}tale double cover $p' \colon \mathcal{F} \rightarrow \mathcal{H}$, with the total space $\mathcal{F}$ given by a smooth curve of genus five, carrying a base-point free involution $\imath' \colon \mathcal{F} \rightarrow \mathcal{F}$. The bielliptic involution $\jmath$ on $\mathcal{H}$ lifts to an involution $\jmath' \colon \mathcal{F} \rightarrow \mathcal{F}$ with eight fixed points, defining a second bielliptic structure $\pi'  \colon \mathcal{F} \rightarrow \mathcal{E}'$, with $\mathcal{E}'$ an elliptic curve that is 2-isogenous to $\mathcal{E}$. 
\beq
\label{d1}
\begin{tikzcd}[row sep=large, column sep=large]
{\mathcal F} 
 \arrow[rr, "p'"]  
 \arrow[d, "\pi'" ']
 \arrow[loop, out=120, in=60, looseness=5,  "j'"] 
&
& \ {\mathcal H} 
 \arrow[rr, "p"]  
 \arrow[d, "\pi"]
 \arrow[loop, out=120, in=60, looseness=5,  "j"] 
 & &  \ \mathcal{C}
\\
\mathcal{E}'
\arrow[rr, "\rm{2-isogeny}"]  
&  & \ \mathcal{E}  
& &  
\end{tikzcd}
\eeq
One has, in this context, a canonical isomorphism $\operatorname{Prym}(\mathcal{F}, p' \colon  \mathcal{F} \rightarrow \mathcal{H}) \cong \operatorname{Jac}(\mathcal{C}')$.
\par Next, we note that the left half of diagram $(\ref{d1})$ is actually a fiber in a one-dimensional family. In order to see this, consider the embedding $\mathcal{C} \hookrightarrow \operatorname{Jac}(\mathcal{C})$, given by a choice of Abel-Jacobi map. The theta divisor $\Theta = [\mathcal{C}]$ gives a principal polarization $\mathscr{U} = \mathcal{O}_{\mathsf{A}}(\Theta)$ on  $\mathsf{A}=\operatorname{Jac}(\mathcal{C})$, which, in turn, establishes a canonical isomorphism $\operatorname{Jac}(\mathcal{C})  \cong  \operatorname{Jac}(\mathcal{C}) ^{\vee}$. Hence, points of order two in $\operatorname{Jac}(\mathcal{C})$ may be viewed as line bundles of order two on $\operatorname{Jac}(\mathcal{C})$. One can then repeat the construction from above, in the context of the Jacobian variety $\operatorname{Jac}(\mathcal{C})$. 
\par First, $\mathscr{L}$ determines a 2-isogeny of abelian surfaces $\Phi \colon \mathsf{B} \rightarrow  \operatorname{Jac}(\mathcal{C})$. The abelian surface $\mathsf{B}$ carries a canonical $(1,2)$-polarization $\mathscr{V} =  \Phi^*(\mathscr{L})$ with $\mathscr{V}^2=4$ and $h^0(\mathscr{V})=2$. The effective divisors for $\mathscr{V}$ form a pencil with four fixed points. Following the work in \cites{MR946234, MR2729013}, a general member of this pencil is, in the generic case, a smooth curve of genus three $\mathcal{D}_t \subset \mathsf{B}$.  The antipodal involution of $\mathsf{B}$ restricts as a bielliptic involution on $\mathcal{D}_t$, the quotient by which gives a double cover $ \pi_t  \colon \mathcal{D}_t \rightarrow \mathcal{E}_t $ mapping on an elliptic curve $\mathcal{E}_t$. One has a canonical isomorphism of abelian surfaces $\operatorname{Prym}( \mathcal{D}_t ,  \pi_t \colon  \mathcal{D}_t \rightarrow \mathcal{E}_t ) \cong \mathsf{B}$. 
\par Second, the pull-back $\Phi^*(\mathscr{L}')$ is a line bundle of order two on $\mathsf{B}$ and, hence, it determines a 2-isogeny $\Phi' \colon \operatorname{Jac}(\mathcal{C}') \rightarrow \mathsf{B}$. The preimage, under $\Phi'$, of each smooth curve $\mathcal{D}_t$ is a smooth curve of genus five $\mathcal{F}_t \subset \operatorname{Jac}(\mathcal{C}') $. As before, the antipodal involution on $\operatorname{Jac}(\mathcal{C}')$ restricts to a bielliptic involution on $\mathcal{F}_t$, leading to a bielliptic structure $\pi'_t  \colon \mathcal{F}_t \rightarrow \mathcal{E}'_t$. 
\beq
\label{d2}
\begin{tikzcd}[row sep=large, column sep=large]
\operatorname{Jac}({\mathcal C}') 
 \arrow[rr, "\Phi'"]  
 \arrow[loop, out=120, in=60, looseness=5,  "-{\rm id}"] 
& & 
\mathsf{B}  
 \arrow[rr, "\Phi"]  
 \arrow[loop, out=120, in=60, looseness=5,  "-{\rm id}"] 
& &  
\operatorname{Jac}(\mathcal{C}) 
\\
{\mathcal F}_t 
\arrow[d, "\pi'_t"]
\arrow[u, hook]
 \arrow[rr, "\rho'_t= \Phi' \vert_{\mathcal{F}_t}"]  
& & 
{\mathcal D}_t 
\arrow[d, "\pi_t"]
\arrow[u, hook]
\\
\mathcal{E}'_t 
\arrow[rr, "\rm{2-isogeny}"]  
& &  \mathcal{E}_t  & &  
\end{tikzcd}
\eeq
The Prym  variety $ \operatorname{Prym}(  \mathcal{F}_t, \rho'_t\colon  \mathcal{F}_t \rightarrow \mathcal{D}_t )$ arises naturally in the above picture, as isomorphic to the Jacobian $\operatorname{Jac}(\mathcal{C}')$. 
\par We note that the curve family $ \mathcal{F}_t$ belongs to the linear system associated with the line bundle $\Phi^{\prime *} \mathscr{V} $, which is of type $(2,2)$ and twice a principal polarization on $\operatorname{Jac}(\mathcal{C}')$.  One has $h^0(\Phi^{\prime *} \mathscr{V})=4$. The family $ \mathcal{F}_t$ is parametrized by a conic curve, within the three-dimensional projective space $\vert \Phi^{\prime *} \mathscr{V} \vert$.
\medskip
\par The goal of this paper is to give an explicit description for the pencils of curves $\mathcal{D}_t$ and $\mathcal{F}_t$. The building block for the entire construction above is simply a choice of a smooth curve of genus two $\mathcal{C}$, as well as a choice of a G\"opel subgroup of $G' \leqslant \operatorname{Jac}(\mathcal{C})[2]$. We shall start with such a curve given explicitly in Rosenhain normal form as
\Beq
\label{eqn:Rosenhain_intro}
 \mathcal{C}: \quad \eta^2 = \xi\,  \big(\xi-1) \, \big(\xi- \lambda_1\big) \,  \big(\xi- \lambda_2 \big) \,  \big(\xi- \lambda_3\big) \,,
\Eeq
such that the ordered tuple $(\lambda_1, \lambda_2, \lambda_3)$ -- with $\lambda_i$ pairwise distinct and different from $(\lambda_4,\lambda_5,\lambda_6)=(0, 1, \infty)$ -- determines a point in the moduli space $\mathfrak{M}$ of curves of genus two with marked level-two structure. A choice of G\"opel subgroup is then equivalent to a choice of $2+2+2$ partition of the six canonical branch points $ \{ \lambda_1, \lambda_2, \lambda_3, \lambda_4, \lambda_5, \lambda_6 \}$. The three Rosenhain $\lambda$-parameters can be expressed as explicit ratios of even Siegel theta constants by Picard's lemma. There are $720$ choices for such expressions: for example, one might use the choice from \cites{MR0141643, MR2367218,MR3712162} to obtain
\Beq
\label{eqn:theta}
\lambda_1 = \frac{\theta_1^2\theta_3^2}{\theta_2^2\theta_4^2} \,, \quad \lambda_2 = \frac{\theta_3^2\theta_8^2}{\theta_4^2\theta_{10}^2}\,, \quad \lambda_3 =
\frac{\theta_1^2\theta_8^2}{\theta_2^2\theta_{10}^2}\,.
\Eeq
We consider the double cover $\mathfrak{M}'$ of $\mathfrak{M}$ given as the set of tuples $(\kappa_{1,5}, \lambda_2, \lambda_3)$ such that $(\lambda_1=\kappa_{1,5}^2,\lambda_2,\lambda_3) \in \mathfrak{M}$.  There is a good reason for the notation $\kappa_{1,5}$, and the reason for it will become apparent later.  For the moment, we only mention that $\kappa_{1,5}$ can be considered a section of a suitable line bundle over $\mathfrak{M}$. We introduce the homogeneous polynomials
\beq
\label{eqn:data_intro}
\begin{split}
  \Delta^{(t)}(X,Y) & = \big(X - t \, Y\big)^2 \,,\\
  r^{(t)}(X, Y)	& = 6\lambda_1 \lambda_2 \lambda_3 \, t^2 X^2  - \big(\lambda_1+ \lambda_2 \lambda_3\big) \big( X^2 +4 t XY + t^2 Y^2)   + 6 Y^2 \,,\\
  r_1^{(t)}(X, Y) 	& = 24   \lambda_1 \lambda_2 \lambda_3 \big( \lambda_1 + \lambda_2\lambda_3 \big) t^2 X^2  + 2 \big( \lambda_1 -5 \lambda_2\lambda_3\big) \big( 5\lambda_1 - \lambda_2\lambda_3\big)  t X Y\\
  & + \big( \lambda_1^2 + \lambda_2^2 \lambda_3^2 -34 \lambda_1 \lambda_2 \lambda_3\big) \big(X^2+ t^2 Y^2\big) + 24  \big( \lambda_1 + \lambda_2\lambda_3 \big) Y^2\,,\\
    p(X,Y)  		& = \big(\lambda_1 X^2-Y^2 \big) \big( \lambda_2 \lambda_3 \, X^2 - Y^2 \big)  \,,
\end{split}  
\eeq
and the parameters $p_0^{(t)}=p(t,1)$ and
\beq
\label{eqn:moduli_gd_intro}
\begin{split}
c_0 & = \, 2  \big( \lambda_1 -5 \lambda_2\lambda_3\big) \big( 5\lambda_1 - \lambda_2\lambda_3\big) \,  \kappa_{1,5} + \lambda_1^3 + \lambda_2^2\lambda_3^2 \\
& \, - \lambda_1^2 \big( 34 \lambda_2 \lambda_3 - 24 (\lambda_2 + \lambda_3) -1 \big) + \lambda_1\lambda_2\lambda_3 \big(\lambda_2\lambda_3 + 24(\lambda_2+ \lambda_3) - 34\big) \,,\\
c_1 & = \,8 \big (\lambda_1 + \lambda_2 \lambda_3 \big) \, \kappa_{1,5} - 2 \big( 6( \lambda_2 +\lambda_3) - \lambda_2 \lambda_3-1\big) \lambda_1 + 2 \big(\lambda_1^2 +  \lambda_2 \lambda_3 \big)\,,\\
c_2 & = \, \lambda_1 + 1 - 2 \, \kappa_{1,5} \,.
\end{split}
\eeq 
We note that $c_2=0$ implies $\lambda_1=1$ and $\mathcal{C}$ is singular.  
\par Let $\mathcal{D}_{t}$ be the pencil  of plane quartic curves in $\mathbb{P}^2=\mathbb{P}(X,Y,Z)$ given by 
\Beq
\label{eqn:genus-three_intro}
 \mathcal{D}_t: \quad p_0^{(t)}  Z^4 +  \Big(c_2 \,  r_1^{(t)} +  c_1 \,  r^{(t)} + c_0 \, \Delta^{(t)} \Big)\, Z^2 + 9 \, \Big(c_1^2- 4 \, c_0c_2\Big)  \, p =0 \,,
\Eeq
with the involution 
\Beq
 \jmath: \quad [X: Y :Z ] \mapsto [X: Y: -Z] \,,
 \Eeq
 and the degree-two quotient map $\pi_t: \mathcal{D}_t \to \mathcal{Q}_t =\mathcal{D}_t/\langle \jmath \rangle$.  We have the following:
\begin{theorem}
\label{main1}
The pencil in Equation~(\ref{eqn:genus-three_intro}) satisfies the following:
\begin{enumerate}
\item for generic $t$, the curve $\mathcal{D}_t$ is a smooth, bielliptic curve of genus three such that the Prym variety $\operatorname{Prym}(\mathcal{D}_t, \pi_t)$ with its natural polarization of type $(1,2)$ is 2-isogenous to the principally polarized Jacobian variety $\operatorname{Jac}(\mathcal{C})$, i.e.,
\Beqn
 \operatorname{Prym}(\mathcal{D}_t, \pi_t)\  \simeq\  \operatorname{Jac}(\mathcal{C}) \,,
\Eeqn
and $\mathcal{D}_t$ embeds into $\operatorname{Prym}(\mathcal{D}_t, \pi_t)$ as a curve of self-intersection four. 
\item for $t^2 = \lambda_1, \lambda_2\lambda_3$, the curve $\mathcal{D}_t$ is a reducible nodal curve isomorphic to $\mathbb{P}^1 \cup \mathcal{C}'$ where $\mathcal{C}'$ is a $(2,2)$-isogenous, smooth curve of genus two such that
\Beqn
 \operatorname{Jac}(\mathcal{C}')\ = \ \operatorname{Jac}(\mathcal{C})/G' \,,
\Eeqn 
where $G' \subset \operatorname{Jac}(\mathcal{C})[2]$  is the G\"opel group associated with the pairing of the Weierstrass points of $\mathcal{C}$ given by $\{ \lambda_1,\lambda_5=1 \}$, $\{ \lambda_2,\lambda_3\}$, $\{\lambda_4=0,\lambda_6=\infty \}$,
\item for $t^2 = \lambda_2, \lambda_1\lambda_3$, and $t^2 = \lambda_3, \lambda_1\lambda_2$, the curve $\mathcal{D}_t$ is a singular, irreducible curve of geometric genus two with one node,
\item for $t^2 = 0, \pm \lambda_1 \lambda_2 \lambda_3, \infty$, the curve $\mathcal{D}_t$ is smooth and hyperelliptic.
\end{enumerate}
\end{theorem}
\par Let $\mathcal{F}_{t}$ be the family of non-hyperelliptic curves of genus five given as the intersection of three quadrics in $\mathbb{P}^4=\mathbb{P}(V,W,X,Y,Z)$ with
\Beq
\label{eqn:genus_5_curve_intro}
\mathcal{F}_t: \quad \left\lbrace \begin{array}{lcl} 
V^2 &=& c_2  e^2 \,  \Delta^{(t)}  + 2 c_2 e \,  r^{(t)}  + c_2 \, r_1^{(t)} \,,\\
W^2 &=& c_2 f^2  \,  \Delta^{(t)}  + 2 c_2 f   \, r^{(t)}  + c_2 \, r_1^{(t)} \,, \\
VW & = & 2 \, p_0^{(t)} Z^2 + c_0 \, \Delta^{(t)}  + c_1 \,  r^{(t)}  + c_2 \, r_1^{(t)} \,,
\end{array} \right. 
\Eeq
and the involution
\beq
\label{eqn:genus_5_involution_intro}
 \imath': \mathbb{P}^4 \to \mathbb{P}^4 \,, \quad [V: W: X: Y:Z ] \mapsto [-V: -W: X: Y:Z ] \,.
\eeq
Here, the parameters $e$ and $f$ are determined by $e+f= c_1/c_2$, $e f = c_0/c_2$; interchanging $e$ and $f$ amounts to the changing the sign $\pm \kappa_{1,5}$ or, equivalently, swapping the two sheets of the double cover $\mathfrak{M}' \to \mathfrak{M}$. We have the following:
\begin{theorem}
\label{main2a}
Each smooth curve $\mathcal{D}_t$ admits an unramified double cover $\rho'_t: \mathcal{F}_t \to \mathcal{D}_t$ with $\mathcal{F}_t$ smooth and bielliptic. The Prym variety $\operatorname{Prym}(\mathcal{F}_t,\rho'_t)$ is canonically isomorphic to the Jacobian of a curve of genus two given by
\Beq
\label{eqn:genus-two-dual_intro}
  \eta^2 =  \Big(\xi - 2 \big( \lambda_1+ \lambda_2\lambda_3\big) \Big) \Big( \big( \xi + \lambda_1 + \lambda_2\lambda_3)^2-36 \lambda_1\lambda_2\lambda_3\Big) \Big( c_2 \xi^2 + c_1 \xi + c_0 \Big) \,,
\Eeq
which is isomorphic to $\mathcal{C}'$ in Theorem~\ref{main1}(2), and $\mathcal{F}_t$ embeds into $\operatorname{Prym}(\mathcal{F}_t,\rho'_t)$ as a curve of self-intersection eight. 
\end{theorem}
Given the marking of a G\"opel group, Equation~(\ref{eqn:genus-two-dual_intro}) can be brought into the form
\beq
\label{eqn:Kovalevaskaya_curve}
 \eta^2 = \Big(\xi^2 - \frac{D^2}{4} \Big) \Big(16 \xi^3 + 4 A \xi^2 + 4 \xi + A -B^2 \Big) \,,
\eeq
commonly referred to as \emph{Kovalevaskaya curve}, where $A, B, D^2$ are interpreted as physical quantities, namely the constants of motions of the Kovalevskaya top. 
\par We also have the following:
\begin{corollary}
\label{main2}
The Jacobian variety $\operatorname{Jac}(\mathcal{D}_t)$ for $t=0, \infty$ is isogenous to the Jacobian $\operatorname{Jac}(\mathcal{H})$ where $\mathcal{H}$ is the bielliptic, hyperelliptic curve of genus three
\Beq
\label{eqn:curveG3_nf_intro}
  \mathcal{H}: \quad \upsilon^2  =  \big(\zeta^2-1) \, \big(\zeta^2- \lambda_1\big) \,  \big(\zeta^2- \lambda_2 \big) \,  \big(\zeta^2- \lambda_3\big)  \,.
\Eeq
\end{corollary}
\begin{remark}
There is a second choice for $\mathfrak{M}'$ given by an extension of the function field of $\mathfrak{M}$ with $\kappa_{2,3}^2 =\lambda_2\lambda_3$ that yields analogous results in Theorem~\ref{main1} and Corollary~\ref{main2a}. In this case, $c_0, c_1, c_2$ are given by Equation~(\ref{eqn:moduli_gd_1}).
\end{remark}
\subsection{Discussion and overview}
Barth studied abelian surfaces with a polarization of type $(1,2)$ and proved their close connection with Prym varieties of smooth, bielliptic curves of genus three \cite{MR946234}. An excellent summary of Barth's construction was given in \cites{Garbagnati08, MR3010125}. Moreover, the fibers of the Prym map were considered in \cites{MR1188194, MR422289, MR875339, MR2406115, MR3781951}. Abelian surfaces with $(1,2)$-polarization were also discussed in \cites{MR2306633, MR2804549, MR2729013} and by the authors in \cites{Clingher:2017aa, Clingher:2018aa, CMS:2019}. An algebraic-geometric approach for studying 2-isogenous abelian surfaces was introduced in \cite{MR2457735}. Bielliptic curves of genus three and abelian surfaces with $(1,2)$-polarization have also appeared as spectral curves of Lax representations of certain algebraic integrable systems and the Kovalevskaya top \cites{MR912838,MR923636, MR990136, MR3798190}. Solving the equations of motion for the Kovalevskaya top is equivalent to a linear flow on an abelian surface with $(1,2)$-polarization. On the other hand, Kovalevskaya presented in her celebrated paper \cite{MR1554772} a separation of variables of the corresponding integrable system using the (hyperelliptic) curve of genus two in Equation~(\ref{eqn:Kovalevaskaya_curve}) whose Jacobian is associated with the integrals of motion of the Kovalevskaya top.  In this article, we will derive explicit normal forms for the pencil of plane, bielliptic curves of genus three (and their unramified double coverings by canonical curves of genus five) such that the Prym variety of its general member is 2-isogenous to the Jacobian of a very general  curve of genus two in $\mathfrak{M}$ (or the Richelot isogenous curve). 
\par The main difficulty in describing explicitly the items of diagram $(\ref{d2})$, in terms of the Rosenhain $\lambda$-parameters, stems from the inherent laboriousness of  computing or describing curves within abelian surfaces. Our approach, which fixes most of this problem, is to push and understand $(\ref{d2})$ to level of the Kummer surfaces.
\Beq
\label{d3}
\begin{tikzcd}[row sep=large, column sep=large]
\operatorname{Jac}({\mathcal C}') 
 \arrow[rr, "\Phi'"]  
 \arrow[loop, out=120, in=60, looseness=5,  "-{\rm id}"] 
 \arrow[ddd, dashrightarrow, bend right=40]
& & 
\mathsf{B}  
 \arrow[rr, "\Phi"]  
 \arrow[loop, out=120, in=60, looseness=5,  "-{\rm id}"] 
  \arrow[ddd, dashrightarrow, bend left=40]
& &   
\operatorname{Jac}(\mathcal{C})  
\\
{\mathcal F}_t 
\arrow[d, "\pi'_t"]
\arrow[u, hook]
 \arrow[rr, "\rho'_t= \Phi' \vert_{\mathcal{F}_t}"]  
& & 
{\mathcal D}_t 
\arrow[d, "\pi_t"]
\arrow[u, hook]
\\
\mathcal{E}'_t 
\arrow[rr, "\rm{2-isogeny}"]  
\arrow[d, hook]
& &  
\mathcal{E}_t  
\arrow[d, hook] 
& &  \\
\operatorname{Kum}\left ( \operatorname{Jac}{\mathcal{C}'}  \right ) 
\arrow[rr, dashrightarrow, "\phi'"]
&
& \ \operatorname{Kum} \left ( \mathsf{B} \right )  
\arrow[rr, dashrightarrow, "\phi"]
& &  \ \operatorname{Kum} \left ( \operatorname{Jac}{\mathcal{C}} \right )
\end{tikzcd}
\smallskip
\Eeq
Using this point of view, as outlined in diagram $(\ref{d3})$, the pencils $\mathcal{E}_t$ and $\mathcal{E}'_t$ correspond to Jacobian elliptic fibrations on the Kummer surfaces $\operatorname{Kum} \left ( \operatorname{Jac}{\mathcal{C}' } \right ) $ and $\operatorname{Kum} \left ( \mathsf{B} \right )$. The rich geometry of these objects is quite well understood, in particular  the sequence of rational maps
\Beq
 \operatorname{Kum}\left ( \operatorname{Jac}{\mathcal{C}'}  \right )  \ \dasharrow  \ \operatorname{Kum} \left ( \mathsf{B} \right ) \ \dasharrow \ \operatorname{Kum} \left ( \operatorname{Jac}\mathcal{C} \right )
\Eeq 
can be described in terms of even-eight curve configurations introduced in \cites{MR2804549, MR0429917}. 

\par This article is structured as follows: in Section~\ref{section1} we establish convenient normal forms for certain abelian surfaces with polarizations of type $(1,1)$, $(1,2)$, $(2,2)$, and their associated Kummer surfaces. In Section~\ref{section2} we construct a pencil of plane, bielliptic curves of genus three and an induced genus-one fibration from the Abel-Jacobi map of a single smooth quartic curve. This quartic curve is determined by the point of order two $p \in \operatorname{Jac}(\mathcal{C})[2]$ and a G\"opel group $G' \ni p$. We then show that the obtained genus-one fibration is isomorphic to a Jacobian elliptic fibration on $\operatorname{Kum}(\mathsf{B})$. We also prove certain properties for the special members of the pencil of curves of genus three, and we construct their unramified coverings by curves of genus five which we also prove to be bielliptic. In Section~\ref{section4} we combine these results to prove Theorem~\ref{main1}, Theorem~\ref{main2a}, and Corollary~\ref{main2}.
\subsection*{Acknowledgments}
We would like to thank the referee for their thoughtful comments and efforts towards improving our manuscript. 
\section{Plane curves and associated K3 surfaces}
\label{section1}
Polarizations on an abelian surface $\mathsf{A}\cong \mathbb{C}^2/\Lambda$ are known to correspond to positive definite hermitian forms $H$ on $\mathbb{C}^2$,  satisfying $E = \operatorname{Im} H(\Lambda,\Lambda) \subset \mathbb{Z}$.  In turn, such a hermitian form determines the first Chern class of a line bundle in the N\'eron-Severi group $\mathrm{NS}(\mathsf{A})$. The bundle itself is then determined only up to a degree zero line bundle.  We will assume that the Picard number $\rho(\mathsf{A})=1$, so that the N\'eron-Severi group of $\mathsf{A}$ is generated by this line bundle \cite{MR2062673}. One may always choose a basis of $\Lambda$ such that $E$ is given  by a matrix $\bigl(\begin{smallmatrix} 0&D\\ -D&0 \end{smallmatrix} \bigr)$ with $D=\bigl(\begin{smallmatrix}d_1&0\\ 0&d_2 \end{smallmatrix} \bigr)$ where $d_1, d_2 \in \mathbb{N}$, $d_1, d_2 > 0 $, and $d_1$ divides $d_2$. The pair $(d_1, d_2)$ gives the {\it type} of the polarization. 
\par Let $\mathcal{C}$ be a smooth curve of genus two. On its Jacobian $\mathsf{A}=\operatorname{Jac}(\mathcal{C})$ the divisor class $\Theta=[\mathcal{C}]$ is an effective divisor such that the hermitian form associated with the line bundle $\mathscr{U}= \mathcal{O}_\mathsf{A}(\Theta)$ is a polarization of type $(1,1)$, also called a principal polarization. We will also consider an abelian surface $\mathsf{B}$ with a $(1,2)$-polarization given by an ample symmetric line bundle $\mathscr{V}$ such that $\mathscr{V}^2 =4$. In this case, the linear system $|\mathscr{V}|$ is a pencil on $\mathsf{B}$ of generically smooth, bielliptic curves of genus three; see~\cite{MR946234}.
\subsection{Abelian and Kummer surfaces with principal polarization}
\label{sec:PPAS}
Let a smooth curve of genus two $\mathcal{C}$ be given in affine coordinates $(\xi,\eta)$ by the Rosenhain normal form 
\beq
\label{eqn:Rosenhain}
 \mathcal{C}: \quad \eta^2 = \xi \,\big(\xi-1) \, \big(\xi- \lambda_1\big) \,  \big(\xi- \lambda_2 \big) \,  \big(\xi- \lambda_3\big) \,.
\eeq 
We denote the hyperelliptic involution on $\mathcal{C}$ by $\imath_\mathcal{C}$. An ordered tuple $(\lambda_1, \lambda_2, \lambda_3)$ -- where the $\lambda_i$ are pairwise distinct and different from $(\lambda_4,\lambda_5,\lambda_6)=(0, 1, \infty)$ -- determines a point in the moduli space $\mathfrak{M}$ of curves of genus two with marked level-two structure.  The Weierstrass points of $\mathcal{C}$ are the six points  $p_i: (\xi,\eta)=(\lambda_i,0)$ for $i=1, \dots,5$, and  the point $p_6$ at infinity. Unless stated otherwise, we assume that $\mathcal{C}$ is a very general curve of genus two.
\par Translations of the Jacobian $\mathsf{A}=\operatorname{Jac}(\mathcal{C})$ by a point of order two of $\mathsf{A}$ are isomorphisms of the Jacobian and map the set of 2-torsion points to itself. In fact, for any isotropic two-dimensional subspace $G' \cong (\mathbb{Z}/2 \mathbb{Z})^2$ of $\mathsf{A}[2]$, also called \emph{G\"opel group}, it is well known that $\mathsf{A}'=\mathsf{A}/G'$ is again a principally polarized abelian surface~\cite{MR2514037}*{Sec.~23}. The corresponding isogeny $\Psi': \mathsf{A} \to \mathsf{A}'$ between principally polarized abelian surfaces has as its kernel $G' \leqslant \mathsf{A}[2]$ and is called a \emph{$(2,2)$-isogeny}. 
\par In the case of the Jacobian of a curve of genus two, every nontrivial 2-torsion point is the difference of Weierstrass points on $\mathcal{C}$. In fact, the sixteen points of order two of $\mathsf{A}=\operatorname{Jac}(\mathcal{C})$ are obtained using the embedding of the curve into the connected component of the identity in the Picard group, i.e., $\mathcal{C} \hookrightarrow \operatorname{Jac}(\mathcal{C}) \cong \operatorname{Pic}^0(\mathcal{C})$ with $p \mapsto [p -p_6]$. We obtain 15 elements $p_{i j} \in \mathsf{A}[2]$ with $1 \le i < j \le 5$ as
\beq
 \label{oder2points}
  p_{i6} = [ p_i - p_6] \; \text{for $1 \le i \le 5$}\,, \qquad 
  p_{ij}=[ p_i + p_j - 2 \, p_6]  \; \text{for $1 \le i < j \le 5$}\,, 
\eeq
and set $p_0=p_{66}= [0]$. For $\{i, j, k, l, m, n\}=\{1, \dots, 6\}$,  the group law on $\mathsf{A}[2]$ is given by the relations
\beq
 \label{group_law}
    p_0 +  p_{ij} =  p_{ij}\,, \quad  p_{ij} +  p_{ij} =  p_{0}\,, \quad 
    p_{ij} + p_{kl} =  p_{mn}, \quad p_{ij} +
    p_{jk} =  p_{ik}\,.
\eeq
The  space $\mathsf{A}[2]$ of 2-torsion points admits a symplectic bilinear form, called the \emph{Weil pairing}. The Weil pairing is induced by the pairing
\beq
 \langle [ p_i - p_j  ] ,[ p_k - p_l] \rangle =\#\{  p_{i}, p_{j}\}\cap \{ p_{k}, p_{l}\} \mod{2} \,,
\eeq
such that the two-dimensional, maximal isotropic subspaces of $\mathsf{A}[2]$ with respect to the Weil pairing are the G\"opel groups. Then, it is easy to check that there are exactly 15 inequivalent G\"opel groups. We will fix a point of order two, say $p=p_{46} \in \mathsf{A}[2]$, and a G\"opel group $G'=\{ 0, p_{15}, p_{23}, p_{46} \} \ni p$. Using the embedding of the curve into the Picard group, we associate $G'$ with the pairing of the Weierstrass points of $\mathcal{C}$ given by $(\lambda_1,\lambda_5=1)$, $(\lambda_2,\lambda_3)$, $(\lambda_4=0,\lambda_6=\infty)$. Using $G'$ we can construct two natural covering spaces of the moduli space $\mathfrak{M}$, namely the set $\mathfrak{M}'_{p_{23}}$ of tuples $(\lambda_1, \kappa_{2,3}, \lambda_3)$ with $\lambda_2 \lambda_3=\kappa_{2,3}^2$ and the set $\mathfrak{M}'_{p_{15}}$ of tuples $(\kappa_{1,5}, \lambda_2, \lambda_3)$ with $\lambda_1=\kappa_{1,5}^2$ such that $(\lambda_1,\lambda_2,\lambda_3) \in \mathfrak{M}$. In turn, both $\mathfrak{M}'_{p_{23}}$ and $\mathfrak{M}'_{p_{15}}$ are covered by the set of tuples $(\kappa_{1,5}, \kappa_{2,3}, \lambda_3)$. Moreover, we introduce the convenient moduli $\Lambda_1 = (\lambda_1 + \lambda_2\lambda_3)/l$, $\Lambda_2 = (\lambda_2 + \lambda_1\lambda_3)/l$, $\Lambda_3 = (\lambda_2 + \lambda_1\lambda_3)/l$ with $l=\kappa_{1,5} \kappa_{2,3}$. The work of the authors in \cite{Clingher:2018aa} proved that $\kappa_{1,5}, \kappa_{2,3}, l$ are rational functions of the Siegel theta functions.
\par In the case $\mathsf{A}=\operatorname{Jac}(\mathcal{C})$ one knows that the $(2,2)$-isogenous abelian surface $\mathsf{A}'=\mathsf{A}/G'$ satisfies $\mathsf{A}' =\operatorname{Jac}(\mathcal{C}')$ for some smooth curve of genus two $\mathcal{C}'$.  The question is how to describe the curve  $\mathcal{C}'$ explicitly. The relationship between the geometric moduli of the two curves was found by Richelot \cite{MR1578135}; see also \cite{MR970659}: if we choose for $\mathcal{C}$ a sextic equation $\eta^2 = f_6(\xi)$, then any factorization $f_6 = A\cdot B\cdot C$ into three degree-two polynomials $A, B, C$ defines a new curve of genus two $\mathcal{C}'$ given by
\beq
\label{Richelot}
 \mathcal{C}': \quad \Delta_{ABC} \cdot \eta^2 = [A,B] \, [A,C] \, [B,C] 
\eeq
where we have set $[A,B] = B \, \partial_\xi A  - A \, \partial_\xi B$ with $\partial_\xi$ denoting the derivative with respect to $\xi$ and $\Delta_{ABC}$ is the determinant of $(A, B, C)$ with respect to the basis $\xi^2, \xi, 1$. 
We have the following:
\begin{proposition}
\label{prop:isog_genus2_curve}
Let $\mathcal{C}$ be the smooth curve of genus two in Equation~(\ref{eqn:Rosenhain}) and  $G'$  be the G\"opel group  $G'=\{ 0, p_{15}, p_{23}, p_{46} \} \leqslant  \operatorname{Jac}(\mathcal{C})[2]$. Over $\mathfrak{M}'_{p}$ with $p \in \{p_{15}, p_{23}\}$ the curve $\mathcal{C}'$ with $\operatorname{Jac}(\mathcal{C}')=\operatorname{Jac}(\mathcal{C})/G'$ is given by
\beq
\label{eqn:curveG2_nf}
 \mathcal{C}': \quad \eta^2 =  \Big(\xi - 2 \big( \lambda_1+ \lambda_2\lambda_3\big) \Big)\Big( \big( \xi + \lambda_1 + \lambda_2\lambda_3)^2-36 \lambda_1\lambda_2\lambda_3\Big)  \Big( c_2 \xi^2 + c_1 \xi + c_0 \Big)   \,,
\eeq
where for $\mathfrak{M}'_{p_{23}}$ we have $\kappa_{2,3}^2=\lambda_2\lambda_3$ and
\beq
\label{eqn:moduli_gd_1}
\begin{split}
c_0 & =  \, 2  \big( \lambda_1 -5 \lambda_2\lambda_3\big) \big( 5\lambda_1 - \lambda_2\lambda_3\big)  \kappa_{2,3} + \big(24 \lambda_2 \lambda_3 + \lambda_2 + \lambda_3\big) \lambda_1^2 \\
& \, +2 \lambda_1 \lambda_2\lambda_3 \big( 12 \lambda_2 \lambda_3 - 17 (\lambda_2 + \lambda_3) +12 \big) + \lambda_2^2\lambda_3^2 \big(\lambda_2 + \lambda_3 + 24\big) \,,\\
c_1 & = \, 8 \big (\lambda_1 + \lambda_2 \lambda_3 \big)  \kappa_{2,3} - 2 \big( 6 \lambda_2 \lambda_3 - \lambda_2 - \lambda_3\big) \lambda_1 + 2 \big(\lambda_2 + \lambda_3-6\big)  \lambda_2 \lambda_3 \,,\\
c_2 & = \, \lambda_2 + \lambda_3 - 2 \, \kappa_{2,3} \,,
\end{split}
\eeq 
and for $\mathfrak{M}'_{p_{15}}$ we have $\kappa_{1,5}^2=\lambda_1$ and
\beq
\label{eqn:moduli_gd_2}
\begin{split}
c_0 & =  \, 2  \big( \lambda_1 -5 \lambda_2\lambda_3\big) \big( 5\lambda_1 - \lambda_2\lambda_3\big)  \kappa_{1,5} + \lambda_1^3 + \lambda_2^2\lambda_3^2 \\
& \, - \lambda_1^2 \big( 34 \lambda_2 \lambda_3 - 24 (\lambda_2 + \lambda_3) -1 \big) + \lambda_1\lambda_2\lambda_3 \big(\lambda_2\lambda_3 + 24(\lambda_2+ \lambda_3) - 34\big) \,,\\
c_1 & = \, 8 \big (\lambda_1 + \lambda_2 \lambda_3 \big)  \kappa_{1,5} - 2 \big( 6( \lambda_2 +\lambda_3) - \lambda_2 \lambda_3-1\big) \lambda_1 + 2 \big(\lambda_1^2 +  \lambda_2 \lambda_3\big) \,,\\
c_2 & = \,  \lambda_1 + 1 - 2 \, \kappa_{1,5} \,,
\end{split}
\eeq 
with $c_1^2-4 c_0c_2=144 \kappa_p^2 (\lambda_2-1)(\lambda_3-1)(\lambda_2-\lambda_1)(\lambda_3-\lambda_1)$.
\end{proposition}
\begin{remark} 
In Proposition~\ref{prop:isog_genus2_curve} it is assumed that the curve $\mathcal{C}$ is smooth and very general. This is necessary to guarantee that the quintic in $\xi$  has distinct roots. For example, $c_2=0$ implies $\lambda_2 =\lambda_3$ and $\mathcal{C}$ not smooth. Moreover, $\lambda_1 =\lambda_2 \lambda_3$ implies $\Delta_{ABC} =0$ in Equation~(\ref{Richelot}) since $\mathcal{C}$ then admits an elliptic involution.
\end{remark}
\begin{proof}
One checks that
\begin{small}
\begin{gather*}
[A,C]=x^2- \lambda_1\,, \qquad [B,C]=x^2- \lambda_2\lambda_3\, , \\
[A,B]=(1+\lambda_1-\lambda_2-\lambda_3) \, x^2-2(\lambda_1-\lambda_2\lambda_3) \, x +\lambda_1\lambda_2+\lambda_1\lambda_3-\lambda_2\lambda_3-\lambda_1\lambda_2\lambda_3 \,,
\end{gather*}
\end{small}%
and $\Delta_{ABC}=\lambda_1-\lambda_2\lambda_3$. We compute its Igusa-Clebsch invariants, using the same normalization as in \cites{MR3712162,MR3731039}.  Denoting the Igusa-Clebsch invariants of the curve of genus two in Equation~(\ref{Richelot}) and Equation~(\ref{eqn:curveG2_nf}) by $[ I_2 : I_4 : I_6 : I_{10} ] \in \mathbb{P}(2,4,6,10)$ and $[ I'_2 : I'_4 : I'_6 : I'_{10} ]$, respectively, one checks that
\beq
[ I_2 : I_4 : I_6 : I_{10} ] = [ r^2 I'_2 \ : \ r^4I'_4 \ : \ r^6I'_6 \ : \ r^{10}I'_{10} ] = [ I'_2 : I'_4 : I'_6 : I'_{10} ] \,,
\eeq
with $r=18(\lambda_1-\lambda_2\lambda_3)\epsilon$ with $\epsilon=\kappa_{2,3}$ for Equation~(\ref{eqn:moduli_gd_1}) and $\epsilon=\kappa_{1,5}$ for Equation~(\ref{eqn:moduli_gd_2}). Since the Igusa-Clebsch invariants for the two curves give the same point in weighted projective space, the claims follows.
\end{proof}
\begin{remark}
\label{rem:isog_curves}
There are exactly three G\"opel groups that contain the fixed element $p_{46} \in \operatorname{Jac}(\mathcal{C})[2]$, namely the groups 
\beq
 G' =\{ 0, p_{15}, p_{23}, p_{46} \} \,, \quad  G''=\{ 0, p_{13}, p_{25},  p_{46} \} \,, \quad G'''=\{ 0,  p_{12}, p_{35}, p_{46} \}  \,,
\eeq
with Richelot isogenous curves of genus two $\mathcal{C}', \mathcal{C}'', \mathcal{C}'''$. Convenient normal forms for $\mathcal{C}'', \mathcal{C}'''$ are obtained from Equations~(\ref{eqn:curveG2_nf}) by interchanging indices $1 \leftrightarrow 2$ or $1 \leftrightarrow 3$, respectively. By construction, the abelian surfaces $\operatorname{Jac}(\mathcal{C}')$, $\operatorname{Jac}(\mathcal{C}'')$, $\operatorname{Jac}(\mathcal{C}''')$ are all principally polarized and $(2,2)$-isogenous to $\operatorname{Jac}(\mathcal{C})$.
\end{remark}
\begin{remark}
\label{rem:dual_Goepel}
The Richelot isogeny in Equation~(\ref{Richelot}) constructs a model for $\mathcal{C}'$ such that the symmetric polynomials of the coordinates of pairs of Weierstrass points are rational over $\mathfrak{M}$. Our model for $\mathcal{C}'$ in Proposition~\ref{prop:isog_genus2_curve} over $\mathfrak{M}'_{p}$ with $p \in \{p_{15}, p_{23}\}$ has in addition two rational Weierstrass points. It was shown in \cite{Clingher:2018aa} that this guarantees that a dual G\"opel group $G \leqslant  \operatorname{Jac}(\mathcal{C}')[2]$ can be constructed from points of order two with rational coefficients over $\mathfrak{M}'_{p}$ that induces the dual $(2,2)$-isogeny $\Psi: \mathsf{A}' \to \mathsf{A} = \mathsf{A}'/G$.
\end{remark}
\par The element $p_{46} \in \operatorname{Jac}(\mathcal{C})[2]$ determines a partition of the six Weierstrass points of $\mathcal{C}$ in Equation~(\ref{eqn:Rosenhain}) into sets of two, four, and all six points. We obtain three double covers of the projective line $\mathbb{P}_{\xi}$ with affine coordinate $\xi$, branched respectively at the marked sets of two, four, and all six points of genus zero, one, and two, respectively. The three double covers have a common double cover $\mathcal{H}$, which is the fiber product over $\mathbb{P}^1$ of any two of the curves.  Equivalently, the point of order two $p_{46}$ determines a divisor $\mathrm{D}$ of degree zero with the associated line bundle $\mathscr{L} = \mathcal{O}_{\mathcal{C}}(\mathrm{D})$ satisfying $\mathscr{L}^{\otimes 2}=\mathcal{O}_\mathcal{C}$. The zero section of the line bundle then determines the unramified double cover $p: \mathcal{H} \to \mathcal{C}$. Moreover, every unramified double cover of a hyperelliptic curve of genus two is obtained in this way \cite{MR990136}*{p.~387} and \cite{MR770932}. The following lemma was proved in \cite{beshaj2014decomposition}*{Thm.~1}:
\begin{lemma}
\label{lem:product_decomp}
The curve $\mathcal{H}$, given by
\beq
\label{eqn:curveG3_nf}
\begin{split}
  & \upsilon^2  =  \big(\zeta^2-1) \, \big(\zeta^2- \lambda_1\big) \,  \big(\zeta^2- \lambda_2 \big) \,  \big(\zeta^2- \lambda_3\big)  \,,
 \end{split} 
 \eeq
 is a hyperelliptic, bielliptic curve of genus three such that its Jacobian is isogenous to the product of a Jacobian of a smooth curve of genus two $\mathcal{C}$ and an elliptic curve $\mathcal{E}$, i.e.,
\beq
 \operatorname{Jac}{(\mathcal{H})} \ \simeq \  \operatorname{Jac}{(\mathcal{C})} \times \mathcal{E} \,,
\eeq 
where $\mathcal{E}$ is the elliptic curve with the $j$-invariant
\beq
\label{eqn:j-inv}
j = \frac{256\big(\sigma_1^2-\sigma_1\sigma_2-3\sigma_1\sigma_3+\sigma_2^2-3\sigma_2+9\sigma_3\big)^3}{(\lambda_1-1)^2(\lambda_2-1)^2(\lambda_3-1)^2(\lambda_1-\lambda_2)^2(\lambda_1-\lambda_3)^2(\lambda_2-\lambda_3)^2} \,,
\eeq
for $\sigma_1 = \lambda_1 + \lambda_2 + \lambda_3$, $\sigma_2= \lambda_1\lambda_2 + \lambda_1 \lambda_3 + \lambda_2\lambda_3$, $\sigma_3=\lambda_1\lambda_2\lambda_3$.
\end{lemma}
We have the following:
\begin{remark}
Within the coarse moduli space $\mathfrak{M}_3$ of curves of genus three, the hyperelliptic locus $\mathfrak{M}^{\mathrm{h}}_3$ is an irreducible five-dimensional sub-variety. We recall that the set of bielliptic curves of genus three $\mathfrak{M}^{\mathrm{be}}_3$ form an irreducible four-dimensional sub-variety of $\mathfrak{M}_3$ \cite{MR932781}. Moreover, it was proven in \cite{MR1816214} that $\mathfrak{M}^{\mathrm{be}}_3$ is rational and $\mathfrak{M}^{\mathrm{be}}_3 \cap \mathfrak{M}^{\mathrm{h}}_3$ is an irreducible, rational sub-variety of $\mathfrak{M}^{\mathrm{be}}_3$ of codimension one. Each isomorphism class $[\mathcal{H}]$ of $\mathfrak{M}^{\mathrm{be}}_3 \cap \mathfrak{M}^{\mathrm{h}}_3$ can be represented as an unramified double covering of a curve of genus two $\mathcal{C}$. Equation~(\ref{eqn:curveG3_nf}) then provides a normal form for $\mathcal{H}$.
\end{remark}
\begin{remark}
The curve in Equation~(\ref{eqn:curveG3_nf}) admits the base-point free involution $\imath \colon \mathcal{H} \rightarrow \mathcal{H}$ with $(\zeta, \upsilon) \mapsto (-\zeta, -\upsilon)$ covering $p \colon \mathcal{H} \rightarrow \mathcal{C}$ with $(\xi, \eta) = (\zeta^2, \zeta\upsilon)$. It also admits the involution $\jmath \colon \mathcal{H} \rightarrow \mathcal{H}$ with $(\zeta, \upsilon) \mapsto (-\zeta, \upsilon)$ covering the double cover $\pi \colon \mathcal{H} \rightarrow {\mathcal E}$ with $\operatorname{Prym}( \mathcal{H},  p \colon  \mathcal{H} \rightarrow \mathcal{C}) \cong \mathcal{E}$. The involutions $\imath$ and $\jmath$ commute, and their composition $\imath \circ \jmath$ defines a hyperelliptic structure on $\mathcal{H}$. 
\end{remark}
\par The quotient $\mathsf{A}/\langle -{\rm id} \rangle$ (where  $-{\rm id}$ is the antipodal involution on an abelian surface $\mathsf{A}$ with $\rho(\mathsf{A})=1$) has sixteen ordinary double points, called the \emph{nodes}. The double points are the images of the points of order two $p_{i j} \in \mathsf{A}[2]$ for $1\le i < j \le 6$. The minimum resolution of  $\mathsf{A}/\langle -{\rm id} \rangle$, denoted by $\operatorname{Kum}(\mathsf{A})$, is a K3 surface known as the \emph{Kummer surface} associated with $\mathsf{A}$. It contains an even set of 16 disjoint rational curves $\mathrm{E}_{i j}$ which are the exceptional divisors introduced in the blow-up process. A second set of 16 disjoint rational curves $\mathrm{T}_{i j}$ are the images of the translates $p_{i j} + \Theta$ of the theta divisor $\Theta = [\mathcal{C}]$ in $\operatorname{Kum}(\mathsf{A})$; they are called \emph{tropes}. The two sets of rational curves, $\{ \mathrm{E}_{i j} \}$ and $\{ \mathrm{T}_{i j}\}$, have a rich symmetry, called the \emph{$16_6$-symmetry} of a Kummer surface. We call the Kummer surface $\operatorname{Kum}(\mathsf{A})$ \emph{generic} if $\mathsf{A}$ has no extra endomorphisms. 
\par For the symmetric product $\mathcal{C}^{(2)}$, the quotient $\mathcal{C}^{(2)}/\langle \imath_\mathcal{C} \times  \imath_\mathcal{C} \rangle$ is realized as a variety in terms of $U=x^{(1)}x^{(2)}$, $X=x^{(1)}+x^{(2)}$, and $Y=y^{(1)}y^{(2)}$ and the affine equation
\beq
\label{kummer_middle}
  Y^2 = U \big(  U  - X +  1 \big)  \prod_{i=1}^3 \big( \lambda_i^2 \, U  -  \lambda_i \, X +  1 \big) \,.
\eeq
The affine variety in Equation~\eqref{kummer_middle} completes to a hypersurface in $\mathbb{P}(1,1,1,3)$ called the \emph{Shioda sextic} \cite{MR2296439} and is birational to the Kummer surface $\operatorname{Kum}(\operatorname{Jac} \mathcal{C})$. In fact, Equation~\eqref{kummer_middle} corresponds to the double cover of the projective plane branched on six lines tangent to a common conic; see \cite{Clingher:2018aa}. Moreover, the corresponding Kummer surface has sixteen rational tropes.
\par It was shown in \cites{Clingher:2017aa, MR4015343} that the Shioda sextic in Equation~\eqref{kummer_middle} equips the Kummer surface $\mathcal{X}=\operatorname{Kum}(\operatorname{Jac} \mathcal{C})$ with a \emph{Jacobian elliptic fibration}, i.e., an elliptic fibration $\pi_\mathcal{X}: \mathcal{X} \to \mathbb{P}^1$ with section $\sigma_\mathcal{X}$ such that $\pi_\mathcal{X} \circ \sigma_\mathcal{X}=\operatorname{id}$. This becomes obvious when bringing Equation~\eqref{kummer_middle} into the equivalent form
\beq
\label{kummer_middle_ell_p_W}
\mathcal{X}: \quad y^2  = x  \left( x-  \frac{u \left( u^2 - u \, \Lambda_3  +1 \right)   \left( \Lambda_1-\Lambda_2 \right)}{ \left( \Lambda_2-\Lambda_3 \right)} \right) \, \left( x- \frac{u\left( u^2 - u \, \Lambda_2  +1 \right)    \left( \Lambda_1-\Lambda_3 \right)}{\left( \Lambda_2-\Lambda_3 \right)} \right)  \,.
\eeq
Here, the section $\sigma_\mathcal{X}$ is given by the point at infinity in each fiber, and a 2-torsion section is $\tau_\mathcal{X}: (x,y)=(0,0)$, and $U=l u=\kappa_{1,5}\kappa_{2,3}u$. Using the Kodaira classification for singular fibers of Jacobian elliptic fibrations \cite{MR0184257}, we have the following:
\begin{lemma}
\label{lem:FibShioda}
Equation~\eqref{kummer_middle_ell_p_W} determines a Jacobian elliptic fibration on the Kummer surface $\mathcal{X}=\operatorname{Kum}(\operatorname{Jac} \mathcal{C})$. Generically, the Weierstrass model has two singular fibers of Kodaira type $I_0^*$ at $u=0, \infty$, six singular fibers of type $I_2$, and the Mordell-Weil group of sections $\operatorname{MW}(\mathcal{X}, \pi_\mathcal{X})=(\mathbb{Z}/2\mathbb{Z})^2 \oplus \langle 1 \rangle$.
\end{lemma}
In the statement above the symbol $\langle m \rangle$ stands for a rank 1 lattice $\mathbb{Z}x$ satisfying $\langle x, x \rangle = m$ with respect to the height pairing.
\begin{proof}
One easily identifies the collection of singular fibers and the torsion part of the Mordell-Weil group. From a comparison with the results in \cite{MR3263663} one then determines the full Mordell-Weil group.
\end{proof}
We make the following:
\begin{remark}
\label{rem:choice_torsion_sections}
The established normal form for the Jacobian elliptic fibration in Equation~(\ref{kummer_middle_ell_p_W}) involves an additional choice: with $(\lambda_4,\lambda_6)=(0, \infty)$ the grouping of the remaining Weierstrass points as $\{ \lambda_1, \lambda_5=1 \}$ and $\{ \lambda_2, \lambda_3 \}$ marks a 2-torsion section, namely $\tau_\mathcal{X}$. This choice is equivalent to selecting $G'$, i.e., one out of three G\"opel groups containing the point $p_{46} \in \operatorname{Jac}(\mathcal{C})[2]$; see Remark~\ref{rem:isog_curves}.
\end{remark}
\subsection{Abelian and Kummer surfaces with \texorpdfstring{$(1,2)$}{(1,2)}-polarization}
\label{ssec:AbSrfc}
Let us also consider abelian surfaces $\mathsf{B}$ with a polarization of type $(d_1,d_2)=(1,2)$ given by an ample symmetric line bundle $\mathscr{V}$ with $\mathscr{V}^2 =4$. 
\par As for principally polarized abelian surfaces, the quotient $\mathsf{B}/\langle -{\rm id} \rangle$ has sixteen ordinary double points and a minimal resolution, denoted by $\operatorname{Kum}(\mathsf{B})$. The double points are again the images of the points of order two $\{ q_0, \dots, q_{15}\}$ on $\mathsf{B}$, and the disjoint rational curves $\{ \mathrm{K}_0, \dots, \mathrm{K}_{15} \}$ are the exceptional divisors introduced in the blow-up process such that $\mathrm{K}_i \circ \mathrm{K}_j = -2 \delta_{i j}$ for $0\le i \le 15$. They are contained in a minimal primitive sub-lattice of the N\'eron-Severi lattice of $\operatorname{Kum}(\mathsf{B})$, known as \emph{Kummer lattice}. In particular, they form an \emph{even set} in the N\'eron-Severi lattice, and the class $\hat{\mathrm{K}} = \frac{1}{2}(\mathrm{K}_0 + \dots + \mathrm{K}_{15} )$ is an element of this lattice with $\hat{\mathrm{K}}^2=-8$. In fact, the N\'eron-Severi lattice $\mathrm{NS}(\operatorname{Kum}\mathsf{B})$ is generated over $\mathbb{Q}$ by the classes $\mathrm{K}_i$, $\hat{\mathrm{K}}$, and one additional class $\mathrm{H}$ with $\mathrm{H}^2=8$ and $\mathrm{H} \circ \mathrm{K}_i=0$ for $0\le i \le 15$.
\par The polarization line bundle $\mathscr{V}$ defines a canonical map $\varphi_{\mathscr{V}}: \mathsf{B} \to \mathbb{P}^{d_1d_2-1}=\mathbb{P}^1$, such that the linear system $|\mathscr{V}|$ is a pencil on $\mathsf{B}$, and each curve in $|\mathscr{V}|$ has self-intersection equal to $4$.  Since we assume $\rho(\mathsf{B})=1$, the abelian surface $\mathsf{B}$ cannot be a product of two elliptic curves or isogenous to a product of two elliptic curves.  It was proven in \cite{MR2062673}*{Prop.~4.1.6, Lemma 10.1.2} that the linear system $|\mathscr{V}|$ has exactly four base points. To characterize these four base points, Barth proved in~\cite{MR946234} that the base points form the group $\operatorname{Tr}(\mathscr{V})=\{p \in \mathsf{B} \mid \, \operatorname{tr}_p^*\mathscr{V}=\mathscr{V}\}\cong (\mathbb{Z}/2\mathbb{Z})^2$ where elements of $\mathsf{B}$ act by translation $\operatorname{tr}_p(x)=x+p$. Thus, the base points have order two on the abelian surface $\mathsf{B}$; we will choose them to be $\{ q_0, q_1, q_2, q_3 \}$. A curve in the pencil $|\mathscr{V}|$ is never singular at any of the base points $\{ q_0, q_1, q_2, q_3 \}$; see \cite{MR2729013}*{Lemma~3.2}. Barth's seminal duality theorem in~\cite{MR946234}  can then be stated as follows:
\begin{theorem}[Barth]
\label{thm:Barth}
In the situation above, let $\mathcal{D} \in |\mathscr{V}|$ be a smooth curve of genus three in the pencil $|\mathscr{V}|$. There exists a bielliptic involution $\jmath$ on $\mathcal{D}$ with degree-two quotient map $\pi: \mathcal{D} \to \mathcal{Q} =\mathcal{D}/\langle \jmath \rangle$ onto an elliptic curve $\mathcal{Q}$ such that $\mathsf{B}$ is naturally isomorphic to the Prym variety $\operatorname{Prym}(\mathcal{D}, \pi )$ and the involution $-{\rm id}$ restrict to $\jmath$.
\par Conversely, if $\mathcal{D}$ is a smooth bielliptic curve of genus three with degree-two quotient map $\pi: \mathcal{D} \to \mathcal{Q} =\mathcal{D}/\langle \jmath \rangle$ then $\mathcal{D}$ is embedded in $\operatorname{Prym}(\mathcal{D}, \pi)$ as a curve of self-intersection four. The Prym variety $\operatorname{Prym}(\mathcal{D}, \pi)$ is an abelian surface with a polarization of type $(1,2)$.
\end{theorem}
\par We will denote the exceptional curves associated with the base points on the Kummer surface $\operatorname{Kum}(\mathsf{B})$ by $\{ \mathrm{K}_0, \mathrm{K}_1, \mathrm{K}_2, \mathrm{K}_3\}$.  The map $\varphi_{\mathscr{V}}: \mathsf{B} \to \mathbb{P}^1$ induces a Jacobian elliptic fibration on $\pi_\mathcal{Y}: \mathcal{Y}=\operatorname{Kum}(\mathsf{B}) \to \mathbb{P}^1$ with section $\sigma_\mathcal{Y}$ as follows: first, a fibration is obtained by blowing up the base points of the pencil $|\mathscr{V}|$. The fibers of this fibration are the strict transform of the curves $\mathcal{D} \in |\mathscr{V}|$ and so the general fiber is a smooth curve of genus three. The involution~$\imath$ lifts to an involution on this fibration whose fixed points are the exceptional curves over $\{ q_0, q_1, q_2, q_3 \}$. We then take as the general fiber of $\pi_\mathcal{Y}$ the quotient of the general fiber of $\phi_{\mathscr{V}}$ by the bielliptic involution. Since a curve in the pencil $|\mathscr{V}|$ is never singular at any of the base points $\{ q_0, q_1, q_2, q_3 \}$, we can take as zero-section $\sigma_\mathcal{Y}$ the exceptional curve over $q_0$ such that the divisor class of the section is $[\sigma_\mathcal{Y}]=\mathrm{K}_0$. Garbagnati \cites{Garbagnati08, MR3010125} proved:
\begin{proposition}[Garbagnati]
\label{prop:Garbagnati}
The fibration $\pi_\mathcal{Y}: \mathcal{Y}=\operatorname{Kum}(\mathsf{B}) \to \mathbb{P}^1$ has twelve singular fibers of Kodaira type  $I_2$ and no other singular fibers. The Mordell-Weil group satisfies $\operatorname{MW}(\mathcal{Y}, \pi_\mathcal{Y} )_{\mathrm{tor}}=(\mathbb{Z}/2\mathbb{Z})^2$ and $\operatorname{rank} \operatorname{MW}(\mathcal{Y}, \pi_\mathcal{Y} )=3$. The smooth fiber class $\mathrm{F}$ with $\mathrm{F}^2=0$ and $\mathrm{F} \circ \mathrm{K}_0$=1 is given by
\beq
 \mathrm{F} = \frac{\mathrm{H}-\mathrm{K}_0 - \mathrm{K}_1 - \mathrm{K}_2 - \mathrm{K}_3}{2} \,.
\eeq
\end{proposition}
The twelve non-neutral components of the reducible fibers of Kodaira type  $A_1$ represent the classes $\mathrm{K}_4, \dots, \mathrm{K}_{15}$ of the Kummer lattice and are not intersected by the class of the zero section given by $\mathrm{K}_0$. In fact, the remaining four classes $\mathrm{K}_i$ with $0 \le i\le3$ satisfy $\mathrm{F}\circ \mathrm{K}_i=1$ and $\mathrm{K}_i \circ \mathrm{K}_j=0$ with $4\le j \le 15$. Thus, they represent sections of the Jacobian elliptic fibration which intersect only neutral components of the reducible fibers, given by the divisor classes $\mathrm{F}-\mathrm{K}_j$ with $0 \le i\le3$ and $4\le j \le 15$.
\par We now construct a Weierstrass model for the fibration in Proposition~\ref{prop:Garbagnati} as follows: Mehran proved in \cite{MR2804549} that there are fifteen distinct isomorphism classes of rational double covers $\phi: \mathcal{Y} \dasharrow \mathcal{X}$ of the Kummer surface $\mathcal{X}=\operatorname{Kum}(\mathsf{A})$ associated with the principal polarized abelian surface $\mathsf{A}=\operatorname{Jac}(\mathcal{C})$, such that the preimage is a Kummer surface $\mathcal{Y}=\operatorname{Kum}(\mathsf{B})$ associated with an abelian surface $\mathsf{B}$ with the polarization of type $(1,2)$. Mehran computed that the branching loci giving rise to these 15 distinct isomorphism classes of double covers are even eights of exceptional curves on the Kummer surface $\operatorname{Kum}(\mathsf{A})$  \cite{MR2804549}*{Prop.~4.2}: each even eight is itself enumerated by a point of order two $p_{i j} \in \mathsf{A}[2]$ with $1 \le i < j \le6$, and given as a sum in the N\'eron-Severi lattice of the form
\beq
\label{eqn:MehranEE}
 \Delta_{p_{ij}} = \mathrm{E}_{i1} + \dots + \widehat{\mathrm{E}_{ij}} + \dots + \mathrm{E}_{i6} + \mathrm{E}_{j1} + \dots + \widehat{\mathrm{E}_{ij}} + \dots + \mathrm{E}_{j6} \,,
 \eeq
where $\mathrm{E}_{ii}=0$, and $\mathrm{E}_{i j}$ are the exceptional divisors obtained by resolving the nodes $p_{i j}$; the hat indicates divisors that are not part of the even eight. Moreover, Mehran proved that each rational map $\phi_\Delta: \operatorname{Kum}(\mathsf{B}) \dashrightarrow \operatorname{Kum}(\mathsf{A})$ branched on such an even eight $\Delta$ is induced by an isogeny $\Phi_\Delta:  \mathsf{B} \to \mathsf{A}$ of abelian surfaces of  degree two and vice versa~\cite{MR2804549}.  We call such an isogeny $\Phi$ a \emph{$(1,2)$-isogeny}. We have the following:
\begin{remark}
\label{rem:EE}
In terms of the $16_6$-configuration, the zero section $\sigma_\mathcal{X}$ of the elliptic fibration in Lemma~\ref{lem:FibShioda} and the 2-torsion section $\tau_\mathcal{X}: (x,y)=(0,0)$ are identified with the tropes $\mathrm{T}_5=\mathrm{T}_{56}$ and $\mathrm{T}_1=\mathrm{T}_{16}$, respectively. The eight non-central components of the two reducible fibers of type $D_4$ in the elliptic fibration in Equation~\eqref{kummer_middle_ell_p_W}  form the even eight $\Delta_{46}$ on $\operatorname{Kum}(\operatorname{Jac}\mathcal{C})$, consisting of the exceptional divisors for the nodes $\{ p_{i4} \}$ and $\{ p_{i6} \}$ with $i= 1, 2, 3, 5$. Their central components are the tropes $\mathrm{T}_4=\mathrm{T}_{46}$ and $\mathrm{T}_6=\mathrm{T}_{66}$ since the fibers are located over $u=0$ and $u=\infty$, respectively. There are exactly six more exceptional divisors from nodes that occur as components of reducible fibers; see \cites{MR2804549, MR2306633}. The situation is depicted in Figure~\ref{fig:fib1}.
\end{remark}
\begin{figure}
  \includegraphics[width=0.9\linewidth]{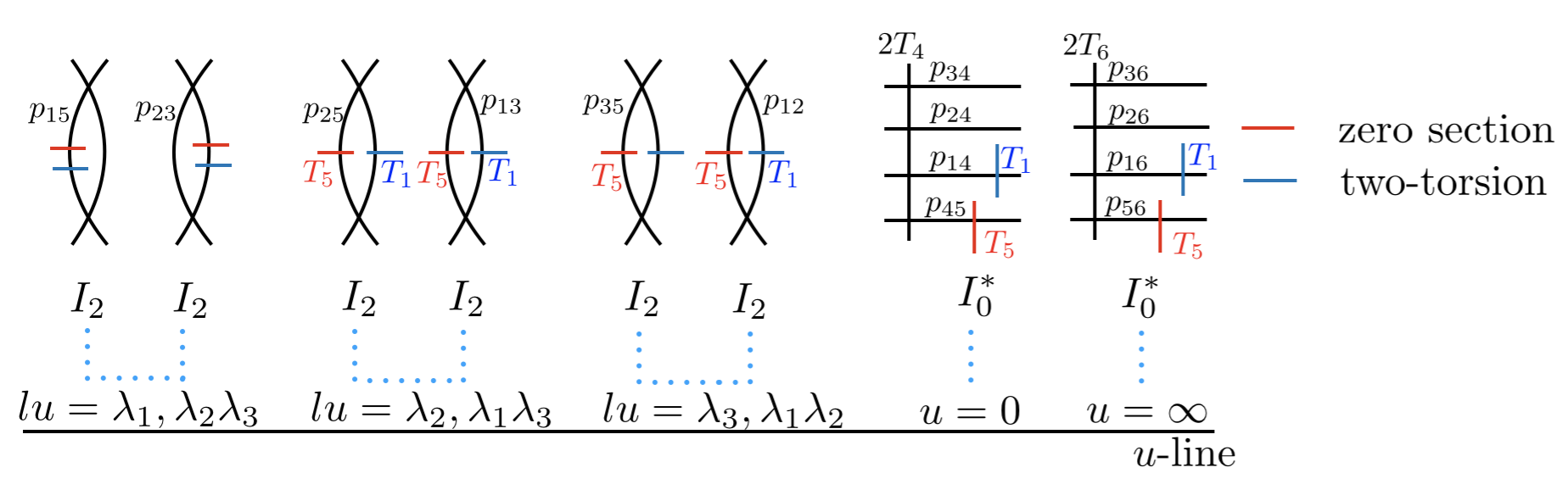}
  \caption{Reducible fibers in Lemma~\ref{lem:FibShioda}}
  \label{fig:fib1}
\end{figure}%
\par Let $\mathcal{Y}=\operatorname{Kum}(\mathsf{B}_{p_{46}})$ be the Kummer surface associated with the abelian surface $\mathsf{B}_{p_{46}}$ with the polarization of type $(1,2)$ induced by
the even eight $\Delta_{p_{46}}$. That is, let $\mathcal{Y}$ be the preimage of the rational double cover $\phi_{\Delta_{p_{46}}}: \mathcal{Y} \dasharrow \mathcal{X}=\operatorname{Kum}(\operatorname{Jac} \mathcal{C})$ branched on the even eight $\Delta_{p_{46}} \subset \operatorname{NS}(\mathcal{X})$. Because of Remark~\ref{rem:EE} the degree-two rational map $\phi_{\Delta_{p_{46}}}$ is induced by the double cover of $\mathbb{P}^1$ branched over $u=0$ and $u=\infty$. We then have
\beq
\label{eqn:phi}
 \phi_{\Delta_{p_{46}}}: \mathcal{Y}\dasharrow \mathcal{X}\,, \quad (v, X, Y)  \mapsto (u, x, y)=(v^2, v^2X, v^3Y) \,.
 \eeq
Accordingly, a Weierstrass equation for $\mathcal{Y}$ is immediately found to be
\beq
\label{eqn:B12}
\mathcal{Y}: \quad Y^2  = X  \left( X-  \frac{ \left( v^4 - v^2  \Lambda_3  +1 \right)   \left( \Lambda_1-\Lambda_2 \right)}{ \left( \Lambda_2-\Lambda_3 \right)} \right)  \left( X- \frac{\left( v^4 - v^2 \Lambda_2  +1 \right)    \left( \Lambda_1-\Lambda_3 \right)}{\left( \Lambda_2-\Lambda_3 \right)} \right),
\eeq
with zero section $\sigma_\mathcal{Y}$ and a 2-torsion section $\tau_\mathcal{Y}: (X,Y)=(0,0)$. According to Mehran's result, there is a corresponding isogeny $\Phi_{\Delta_{p_{46}}}:  \mathsf{B}_{p_{46}} \to \mathsf{A}$ which induces $\phi_{\Delta_{p_{46}}}$. We have the following:
\begin{proposition}
\label{lem:B12}
Equation~\eqref{eqn:B12} determines the Jacobian elliptic fibration on the Kummer surface $\mathcal{Y}=\operatorname{Kum}(\mathsf{B}_{p_{46}})$ from Proposition~\ref{prop:Garbagnati}. Generically, the Weierstrass model has 12 singular fibers of Kodaira type $I_2$, and the Mordell-Weil group $\operatorname{MW}(\mathcal{Y}, \pi_\mathcal{Y})=(\mathbb{Z}/2\mathbb{Z})^2 \oplus \langle 1 \rangle^{\oplus 2} \oplus \langle 2 \rangle$.
\end{proposition}
\begin{proof}
One easily identifies the collection of singular fibers and the torsion part of the Mordell-Weil group. A complete set of generators for the Mordell-Weil group was provided in \cites{MR3995925, CMS:2019}. In \cites{MR3995925} three pairwise orthogonal, non-torsion sections of the elliptic fibration $(\pi_\mathcal{Y}, \sigma_\mathcal{Y})$ of minimal height were constructed that generate a rank-three sub-lattice of the Mordell-Weil group of sections. It was proved in \cite{Garbagnati08}*{Prop.~\!2.2.4} that the transcendental lattice of the Kummer surface $\operatorname{Kum}(\mathsf{B})$ with polarization of type $(1,2)$ is isometric to $U(2)\oplus U(2) \oplus \langle -8 \rangle$ and  the determinant of the discriminant form equals $2^7$ where $U$ is the standard rank-two hyperbolic lattice. This is in agreement with the determinant of the discriminant form for the N\'eron-Severi lattice obtained from an elliptic fibration with section, twelve singular fibers of Kodaira type  $I_2$, and a Mordell-Weil group of sections $(\mathbb{Z}/2\mathbb{Z})^2 \oplus \langle 1 \rangle^{\oplus 2} \oplus \langle 2 \rangle$.
\end{proof}
\subsection{Kummer surfaces with \texorpdfstring{$(2,2)$}{(2,2)}-polarization}
On $\mathcal{Y}=\operatorname{Kum}(\mathsf{B}_{p_{46}})$ we can construct another even eight of exceptional curves $\Delta'$ as follows: the fibration in Equation~(\ref{eqn:B12}) has eight reducible fibers of type $A_1$ where the 2-torsion section $\tau_\mathcal{Y}: (X,Y)=(0,0)$ intersects the non-neutral component, i.e., the component of the fiber not met by the zero-section $\sigma_\mathcal{Y}$. These divisors from an even eight which is precisely the even eight $\Delta'=\Delta_{G'}$ determined by the G\"opel group $G'=\{ 0, p_{15}, p_{23}, p_{46} \} \subset \operatorname{Jac}(\mathcal{C})[2]$, namely the union of the non-neutral components of the preimages of the four reducible $A_1$-fibers in the fibration~(\ref{kummer_middle_ell_p_W}) on $\operatorname{Kum}(\operatorname{Jac} \mathcal{C})$ not containing $p_{15}$, $p_{23}$ under the double cover $\phi_{\Delta_{p_{46}}}: \mathcal{Y} \dasharrow \mathcal{X}=\operatorname{Kum}(\operatorname{Jac} \mathcal{C})$; see Figure~\ref{fig:fib1}.  We then obtain a new K3 surface $\mathcal{X}'$ as the preimage of the rational double cover $\phi_{\Delta_{G'}}: \mathcal{X}' \dasharrow \mathcal{Y}$ branched on $\Delta_{G'} \subset \operatorname{NS}(\mathcal{Y})$. Since the even eight consists only of non-neutral components of reducible fibers, the new K3 surface $\mathcal{X}'$ carries an induced elliptic fibration with section and 2-torsion section. In fact, using the results in \cite{MR3995925} a Weierstrass model for $\mathcal{X}'$ is found to be
\begin{equation}
\label{XXsurface}
\begin{split}
 & \qquad  \mathcal{X}': \qquad  \, y^2  = x^3  + (v^2 + v^{-2} - \Lambda_1)^2 \, v^4 x \\
+&\; \frac{(2\Lambda_1-\Lambda_2-\Lambda_3)(v^2+v^{-2}) + 2 \Lambda_2\Lambda_3 -\Lambda_1\Lambda_2-\Lambda_1\Lambda_3}{\Lambda_2-\Lambda_3} \, v^2x^2\,,
\end{split} 
\end{equation}
with zero section $\sigma_{\mathcal{X}'}$ and 2-torsion section $\tau_{\mathcal{X}'}: (x,y)=(0,0)$.  Thus, we are in the situation where both K3 surfaces $\mathcal{X}'$ and $\mathcal{Y}$ are endowed with  Jacobian elliptic  fibrations which, in addition to trivial sections, each carry a section that makes an element of order two in the Mordell-Weil group. Fiberwise translations by these 2-torsion sections are then known to define involutions $\imath_{\mathcal{X}'}$ on $\mathcal{X}'$ and $\imath_\mathcal{Y}$ on $\mathcal{Y}$, respectively, called \emph{van~Geemen-Sarti involutions} \cites{MR2274533,MR2824841}. The involutions are special Nikulin involutions, and from the Nikulin construction we obtain a pair of dual geometric 2-isogenies between $\mathcal{X}'$ and $\mathcal{Y}$:
\Beq
\label{isog_intro}
 \xymatrix 
{ \mathcal{Y} \ar @(dl,ul) ^{\imath_{\mathcal{Y}}}
\ar @/_0.5pc/ @{-->} _{\phi''} [rr] &
& \mathcal{X}' \ar @(dr,ur) _{\imath_{\mathcal{X}'}}
\ar @/_0.5pc/ @{-->} _{\phi'} [ll] \\
} 
\Eeq
We have the following:
\begin{proposition}
\label{lem:XX}
Equation~\eqref{XXsurface} determines a Jacobian elliptic fibration on the Kummer surface $\mathcal{X}'=\operatorname{Kum}(\operatorname{Jac} \mathcal{C}')$ for $\mathcal{C}'$ given in Proposition~\ref{prop:isog_genus2_curve}. Generically, the Weierstrass model has 4 singular fibers of Kodaira type $I_4$, 8 singular fibers of type $I_1$, and the Mordell-Weil group $\operatorname{MW}(\mathcal{X}', \pi_{\mathcal{X}'})=\mathbb{Z}/2\mathbb{Z} \oplus \langle 1 \rangle^{\oplus 3}$. 
\end{proposition}
\begin{proof}
Rosenhain moduli $\Lambda'_1, \Lambda'_2, \Lambda'_3$ for the curve of genus two $\mathcal{C}'$ in Proposition~\ref{prop:isog_genus2_curve} were computed as rational functions of the moduli $\Lambda_1, \Lambda_2, \Lambda_3$ of $\mathcal{C}$ and vice versa in \cite{Clingher:2018aa}. Substituting these relations into Equation~\eqref{XXsurface}, one recovers the Weierstrass model of the elliptic fibration {\tt (7)} in the list of all elliptic fibrations on $\operatorname{Kum}(\operatorname{Jac} \mathcal{C}')$ in~\cite{MR3263663}*{Thm.~2}.
\end{proof}
\par We have the following:
\begin{remark}
In the situation above, it follows $\phi'=\phi_{\Delta_{G'}}$, i.e., the rational double cover branched on $\Delta_{G'} \subset \operatorname{NS}(\mathcal{Y})$ is precisely the 2-isogeny covered by the van~Geemen-Sarti involution $\imath_{\mathcal{X}'}$. On the one hand, the even eight $\Delta'=\Delta_{G'}$ determined by the G\"opel group $G'=\{ 0, p_{15}, p_{23}, p_{46} \}  \leqslant  \operatorname{Jac}(\mathcal{C})[2]$ as the union of the non-neutral components in the preimages of the four reducible $A_1$-fibers in the fibration~(\ref{kummer_middle_ell_p_W}) on $\mathcal{X}=\operatorname{Kum}(\operatorname{Jac} \mathcal{C})$ not containing $p_{15}$, $p_{23}$ under $\phi_{\Delta_{p_{46}}}: \mathcal{Y} \dasharrow \mathcal{X}$. On the other hand, the van~Geemen-Sarti involution $\imath_{\mathcal{X}'}$ was the fiberwise translation by the 2-torsion section $\tau_\mathcal{Y}: (X,Y)=(0,0)$ in the fibration~(\ref{eqn:B12}) on $\mathcal{Y}=\operatorname{Kum}(\mathsf{B})$ which in turn was determined by the G\"opel group $G'$ as well; see Remark~\ref{rem:choice_torsion_sections}.
\end{remark}
\par We also make the following:
\begin{remark}
\label{rem:choice_even_eight}
It was shown in \cite{Clingher:2017aa} that the dual isogeny $\phi''$ in Equation~(\ref{isog_intro}) is branched on the even eight of exceptional curves $\Delta_{p'_{46}} \subset \operatorname{NS}(\operatorname{Kum}(\operatorname{Jac} \mathcal{C}'))$. Accordingly, $\mathcal{Y}$ is the Kummer surface associated with two different abelian surfaces with a polarization of type $(1,2)$. One is obtained from the double cover of $\operatorname{Kum}(\operatorname{Jac}\mathcal{C})$ branched on $\Delta_{p_{46}}$, the other from the double cover of $\operatorname{Kum}(\operatorname{Jac}\mathcal{C}')$ branched on $\Delta_{p'_{46}}$. Thus, we have $\mathcal{Y} \cong \operatorname{Kum}(\mathsf{B}_{p_{46}}) \cong \operatorname{Kum}(\mathsf{B}_{p'_{46}})$.
\end{remark}
\section{Abel-Jacobi map, canonical curves, and associated K3 surfaces}
\label{section2}
In this section we will construct a pencil of plane, bielliptic curves of genus three and its induced genus-one fibration from the Abel-Jacobi map of a single smooth quartic curve. We then show that the obtained genus-one fibration always admits four rational sections and is isomorphic to a Jacobian elliptic fibration on a K3 surface of Picard-rank 17. We also prove certain properties for special members of the pencil and the close relation to a linear system of quadrics in $\mathbb{P}^4$.
\subsection{The Abel-Jacobi map}
\label{ssec:AJM}
Let $\mathcal{Q}$ be a smooth curve of genus one given by the quartic equation $w^2 = P(x) = \sum_{i=0}^4 a_i x^{4-i}$, using the affine coordinates $(x, w) \in \mathbb{C}^2$.  Given a point $(x_0,-w_0) \in \mathcal{Q}$ we consider the Abel-Jacobi map $J_{(x_0,-w_0)} : \mathcal{Q} \to \operatorname{Jac}(\mathcal{Q})$ which relates the algebraic curve $\mathcal{Q}$ to its Jacobian variety $\operatorname{Jac}(\mathcal{Q})$, i.e., an elliptic curve. A classical result due to Hermite states that $\operatorname{Jac}(\mathcal{Q}) \cong \mathcal{E}$ where $\mathcal{E}$ is the elliptic curve given by
\beq
\label{eqn:EC}
 \mathcal{E}: \quad \eta^2 = S(\xi) = \xi^3 + f\,  \xi + g \,.
\eeq
Here, we are using the affine coordinates $(\xi,\eta) \in \mathbb{C}^2$ and
\beq
\label{eqn:hermite}
 f = - 4 a_0 a_4 + a_1 a_3 - \frac{1}{3} a_2^2 \,, \qquad g = -\frac{8}{3} a_0 a_2 a_4 + a_0 a_3^2 + a_1^2 a_4 - \frac{1}{3} a_1 a_2 a_3 + \frac{2}{27} a_2^3 \,;
\eeq
the construction was reviewed in \cites{MR2166182, MR3995925}.  We introduce the polynomial
\beq
 R(x, x_0)  = a_4 \,x^2 x_0^2 + \frac{a_3}{2} x x_0 \big( x + x_0\big) + \frac{a_2}{6}  \big( x^2 + x_0^2\big)  + \frac{2 a_2}{3} x x_0 + \frac{a_1}{2} \big( x+x_0\big) + a_0  \,,\\
\eeq
such that $R(x, x)=P(x)$. It turns out that the polynomial $P(x) P(x_0) - R(x, x_0)^2$ factors. There is a polynomial $R_1(x, x_0)$ of bi-degree $(2,2)$ such that 
\beq
\label{eqn:relat}
 \forall x, x_0: \quad R(x, x_0)^2 +  R_1(x, x_0) \, \big(x - x_0\big)^2 - P(x) \, P(x_0) =0 \,,
\eeq
and we set $Q(x)=R_1(x, x)$. In particular, we have 
\beq
 Q(x)= \frac{1}{3} P(x) P''(x) - \frac{1}{4} P'(x)^2 \,.
\eeq
We denote the discriminants of $\mathcal{Q}$ and $\mathcal{E}$ by $\Delta_\mathcal{Q} =\operatorname{Discr}_x(P)$ and $\Delta_\mathcal{E} =\operatorname{Discr}_\xi(S)$, respectively, such that $\Delta_\mathcal{Q} =\Delta_\mathcal{E}$ by construction. One also checks $\operatorname{Discr}_x(Q)=S(0)^2\operatorname{Discr}_x(P)$. From now on, we will assume that 
\beq
\label{eqn:constraint0}
 \operatorname{Discr}_x(Q) =S(0)^2\operatorname{Discr}_x(P) \not =0 \,.
\eeq 
As before, we also set $[P, Q]= \partial_xP \cdot Q  - P \cdot \partial_xQ$. A tedious but straightforward computation yields the following:
\begin{lemma}
\label{lem:AJM}
For a smooth curve of genus one $\mathcal{Q}$ given by $w^2 = \sum_{i=0}^4 a_i x^{4-i}$, the Abel-Jacobi map $J_{(x_0,-w_0)} : \mathcal{Q} \to \mathcal{E} \cong  \operatorname{Jac}(\mathcal{Q})$ maps $(x,y) \mapsto (\xi,\eta)$ with
\beq
\label{eqn:AJM}
 \xi = 2 \frac{ R(x, x_0) - w w_0}{(x-x_0)^2} \,, \quad \eta =  \frac{4 w w_0 (w-w_0)}{(x-x_0)^3} - \frac{P'(x) w_0+ P'(x_0) w}{(x-x_0)^2} \quad \text{for $x\not = x_0$} \,, 
\eeq 
the point $(x_0,-w_0) \in \mathcal{Q}$ to the point at infinity on $\mathcal{E}$, and $(x_0,w_0)$ to the point with $\xi=-Q(x_0)/P(x_0)$, $\eta= [P, Q]_{x_0}/(2w_0^3)$ if $w_0 \not = 0$. 
\end{lemma}
\begin{remark}
Equation~(\ref{eqn:EC}) is independent of the chosen point $(x_0,-w_0)$. Thus, the Jacobian elliptic curve of a quartic curve exists independently of whether the quartic itself admits a rational point.
\end{remark}
\par  It follows easily from Equation~(\ref{eqn:AJM}) that the coordinates $x$ and $\xi$ in the Abel-Jacobi map  $(\xi,\eta) = J_{(x_0,-y_0)} (x,y)$ are related by the bi-quadratic polynomial
\beq
\label{eqn:correspondence}
  \xi^2  (x - x_0)^2  - 4 \, \xi R(x, x_0)  -  4 R_1(x, x_0)   =0 \,.
\eeq 
This equation defines an algebraic correspondence between points of the two projective lines with affine coordinates $\xi$ and $x$, respectively, where -- given a point $x$ -- there are two solutions for $\xi$  in Equation~(\ref{eqn:correspondence}) and vice versa. 
\subsection{Associated K3 surfaces}
\label{ssec:AJ_K3}
We now construct a family of curves of genus one $\mathcal{Q}_{x_0}$ over the projective line $\mathbb{P}_{x_0}$ (with affine coordinate $x_0$) from two copies of Equation~(\ref{eqn:correspondence}).  Let the curves of genus one $\mathcal{Q}_{x_0}$ be given by
\beq
\label{eqn:genus-one}
   \mathcal{Q}_{x_0}: \quad w^2 = q_1(x, x_0) \, q_2 (x, x_0) \,,
\eeq
where $q_1, q_2$ are the two conics $\mathrm{C}_i = \mathrm{V}(q_i)$ for $i=1, 2$ with
\beq
\begin{split}
 q_1 & = \gamma^2  (x - x_0)^2  - 4 \gamma R(x, x_0)  - 4 R_1(x, x_0)\,, \\
 q_2 & = \delta^2  (x - x_0)^2  - 4 \delta R(x, x_0)  - 4 R_1(x, x_0) \,.
\end{split} 
\eeq
Thus, the general element $\mathcal{Q}_{x_0}$ is the double cover $\chi: \mathcal{Q}_{x_0} \to \mathbb{P}^1$ of the projective line $\mathbb{P}^1$ (with affine coordinate $x$) branched on points $x'_n$ with $n=1,\dots, 4$ satisfying
\beq
\label{eqn:branching}
  \xi^2  (x'_n - x_0)^2 - 4\, \xi R(x'_n, x_0) - 4 R_1(x'_n, x_0)    =0 \,,
\eeq 
where $n=1,2$ and $n=3,4$ correspond to the solutions of Equation~(\ref{eqn:branching}) for $\xi=\gamma$ and $\xi=\delta$, respectively.  We denote the four ramification points of $\chi$ by $p'_n: (x, w) =(x'_n,0) \in \mathcal{Q}_{x_0}$ with $1\le n \le 4$.
\par We also introduce $\mathcal{Z}= \coprod_{x_0} \mathcal{Q}_{x_0}$, i.e., the total space of the genus-one fibration obtained by varying the parameter $x_0$ in Equation~(\ref{eqn:genus-one}). The discriminant of the fiber is easily checked to be a polynomial of degree 24, namely
\beq
\label{eqn:discriminant}
\Delta_\mathcal{Z} = 2^{20}\nu^2 (\mu^2-\nu\kappa)  P(x_0)^2 \, \big( \kappa \,  P(x_0)^2 + 2 \mu \, P(x_0) Q(x_0) + \nu \, Q(x_0)^2\big)^2 \,.
\eeq
It follows that the minimal resolution of the total space $\mathcal{Z}$ is an elliptic K3 surface (not necessarily with section) with an obvious projection map $\pi_\mathcal{Z}: \mathcal{Z} \to \mathbb{P}_{x_0}$.  Here, we have set
\beq
\label{eqn:moduli}
\begin{split}
 & \kappa  = \frac{(\gamma \delta)^2}{2} - \gamma\delta \, S'(0) - 2 (\gamma+\delta) \, S(0)+ \frac{S'(0)^2}{2}\,, \\
 \mu  =& \, \frac{\gamma\delta(\gamma+\delta)}{2} + \frac{(\gamma+\delta)}{2} \, S'(0)+ S(0)\,,\qquad
 \nu  = \frac{(\gamma-\delta)^2}{2} \,,
\end{split} 
\eeq 
with $S(\xi)$ given in Equation~(\ref{eqn:EC}). For $\gamma=\delta$, the curve $\mathcal{Q}_{x_0}$ is reducible. Hence, we will always assume that $\gamma, \delta \in \mathbb{C}$ are chosen such that
\beq
\label{eqn:constraints}
 \nu \not = 0 \,,\qquad \mu^2-\nu\kappa \not =0 \,.
\eeq
If we consider two pairs of points $\pm q_\gamma,  \pm q_\delta \in \mathcal{E}$ on the elliptic curve in Equation~(\ref{eqn:EC}) with coordinates $(\xi_{q_\gamma} = \gamma, \, \pm \eta_{q_\gamma})$ with $\eta_{q_\gamma}^2 = S(\gamma)$ and $(\xi_{q_\delta} = \delta, \, \pm \eta_{q_\delta})$ with $\eta_{q_\delta}^2 = S(\delta)$, respectively, we find $ \mu^2-\nu\kappa= \eta_{q_\gamma}^2\eta_{q_\delta}^2$.  The constraint $\mu^2-\nu\kappa \not =0$ implies that neither $q_\gamma$ nor $q_\delta$ is a 2-torsion point of $\mathcal{E}$, i.e., $2q_\gamma, 2q_\delta \not = 0$ where $0  \in \mathcal{E}$ denotes the neutral element of the elliptic curve. Thus, the constraints in Equation~(\ref{eqn:constraints}) are equivalent to requiring
\beq
\label{eqn:constraintsEC}
 q_\gamma \not=  \pm q_\delta  \,, \qquad \text{and} \qquad q_\gamma+q_\delta \not =  \pm \big(  q_\gamma-q_\delta \big)\,.
\eeq
We have the following crucial lemma:
\begin{lemma}
\label{lem:sections}
The elliptic fibration $\pi_\mathcal{Z}: \mathcal{Z} \to \mathbb{P}_{x_0}$ admits four sections -- rational over $\mathbb{C}(x_0)$ -- given by $p''_n:  (x, w) = ( x''_n, B(x''_n, x_0))$  with $1\le n \le 4$  and
\beq  
\label{eqn:def_B}
 B(x, x_0) =\gamma \delta  \,(x-x_0)^2 - 4 \, R_1(x, x_0) - 2 (\gamma+\delta) \, R(x, x_0) \,, 
\eeq
where $\{x''_n\}_{n=1}^4$ are the roots of $P(x)=0$. 
\end{lemma}
\begin{proof}
The proof follows by checking that $(x, w) = ( x''_n, B(x''_n, x_0))$ is a polynomial solution of Equation~(\ref{eqn:genus-one}) for $1\le n \le 4$ 
\end{proof}
\begin{proposition}
\label{prop:JAC}
The elliptic fibration $\pi_\mathcal{Z}: \mathcal{Z} \to \mathbb{P}_{x_0}$ is birationally equivalent to a Jacobian elliptic K3 surface with a Weierstrass model given by
\beq
\label{eqn:JacFib}
 Y^2 = X^3 - 2 \Big(\mu P(x_0) + \nu Q(x_0)  \Big) \, X^2 + \Big( \big(\mu P(x_0) + \nu Q(x_0)  \big)^2 - \big(\mu^2-\kappa \nu\big) P(x_0)^2 \Big) X \,,
\eeq
with zero section $\sigma_\mathcal{Z}$ and a 2-torsion section $\tau_\mathcal{Z}: (X,Y)=(0,0)$. Generically, the Weierstrass model has 12 singular fibers of Kodaira type $I_2$, and a Mordell-Weil group with $\operatorname{MW}(\mathcal{Z}, \pi_\mathcal{Z})_{\mathrm{tor}}=(\mathbb{Z}/2\mathbb{Z})^2$ and $\operatorname{rank} \operatorname{MW}(\mathcal{Z}, \pi_\mathcal{Z}) = 3$.
\end{proposition}
\begin{proof}
A Weierstrass model for the Jacobian $\operatorname{Jac}(\mathcal{Q}_{x_0})$ of the curve of genus one $\mathcal{Q}_{x_0}$ can be constructed using Hermite's equations in Section~\ref{ssec:AJM}. Accordingly, the minimal resolution of the total space $\mathcal{Z}'' = \coprod_{x_0} \operatorname{Jac}(\mathcal{Q}_{x_0})$ is a Jacobian elliptic K3 surface $(\mathcal{Z}'', \pi_{\mathcal{Z}''}, \sigma_{\mathcal{Z}''})$ where $\pi_{\mathcal{Z}''}: \mathcal{Z}'' \to \mathbb{P}_{x_0}$ is the projection map and the section $\sigma_{\mathcal{Z}''}$ is given by the smooth point at infinity in each fiber. One also checks that the discriminant of the Jacobian elliptic fibration is given by $\Delta_{\mathcal{Z}''} = 2^{-18} \Delta_\mathcal{Z}$.  The resulting equation is easily seen to admit three 2-torsion sections $T_1, T_2, T_3$ (as we vary $x_0$), and accordingly the equation can be brought into the form of Equation~(\ref{eqn:JacFib}). The torsion sections are given by $T_1=\tau_\mathcal{Z}: (X,Y)=(0,0)$ and
\beqn
T_{2,3}: \quad (X,Y)= \Big((-\mu \pm \sqrt{\mu^2 - \kappa\nu}) P(x_0) - \nu Q(x_0),0\Big) \,.
\eeqn 
\par  For $P(x)=\prod_n (x-x''_n)$, the fibration in Equation~(\ref{eqn:genus-one}) has four rational sections given by $(x, w)=(x''_n, w''_n)$ with $1 \le n \le 4$ where $w''_n=B(x''_n, x_0)$ is the polynomial in $x_0$. The existence of at least one rational section implies an isomorphism $\mathcal{Z}'' \cong \mathcal{Z}$ as elliptic K3 surfaces; see \cite{MR3995925}*{Thm.~3.4}. This can be seen as follows: we consider Equation~(\ref{eqn:genus-one}) -- when expanded in terms of $x$ -- an equation of the form
\beq
  \mathcal{Q}_{x_0}: \quad w^2 = \tilde{a}_0(x_0) \, x^4 + \dots +  \tilde{a}_4(x_0)  \,,
\eeq
with polynomials $\tilde{a}_i(x_0)$ of degree four. On $\mathcal{Q}_{x_0}$ the point $(x, w)=(x''_1, w''_1)$ is a rational point for every $x_0$ and can be used to construct a (fiberwise) Abel-Jacobi map as in Lemma~\ref{lem:AJM}. The Abel-Jacobi map, viewed as a birational map $(x_0, x, w) \mapsto (x_0, X, Y)$, then induces the isomorphism $\mathcal{Z}'' \cong \mathcal{Z}$ and maps $(x''_1, w''_1)$ to the section $\sigma_\mathcal{Z}$, given by the point at infinity in each fiber. One checks that this isomorphism maps the other three sections $(x, w)=(x''_n, w''_n)$ for $n=2,3,4$ to the non-torsion sections $S_{m}: (X,Y)=(X''_{m}(x_0) , Y''_{m}(x_0))$ with $m=n-1$ where $X''_{m}(x_0)$ and $Y''_{m}(x_0)$ are certain polynomials with coefficients in $\mathbb{Q}[\gamma, \delta, x_1'', \dots x''_4]$ of degree four and six, respectively. The same computation as in \cite{CMS:2019} then shows that the three sections $S_m$ are combinations of sections of minimal height that generate a Mordell-Weil group with $\operatorname{rank} \operatorname{MW}(\mathcal{Z}, \pi_\mathcal{Z}) = 3$.
\end{proof}
\par For two arbitrary sections $S'$ and $S''$ of a Jacobian elliptic fibration, one defines the \emph{height pairing} using the formula
\beq
   \langle S', S'' \rangle = \chi^{\text{hol}} + \sigma_\mathcal{Z} \circ S' + \sigma_\mathcal{Z} \circ S'' - S' \circ S'' - \sum_{\{x_0|\Delta_\mathcal{Z}=0\}} C_{x_0}^{-1}(S',S'') \,,
\eeq
where the holomorphic Euler characteristic is $\chi^{\text{hol}}=2$, and $C_{x_0}^{-1}$ is the inverse Cartan matrix of the reducible fiber at $x_0$. In our case, $C_{x_0}^{-1}$ is the inverse Cartan matrix of a fibre of type $A_1$ located over the points $x_0$ with $\Delta_\mathcal{Z}=0$ in Equation~(\ref{eqn:discriminant}) and contributes $(\frac{1}{2})$ if and only if both $S'$ and $S''$ intersect the non-neutral component of this fiber, i.e., the component not met by the zero-section $\sigma_\mathcal{Z}$. The non-neutral components constitute twelve rational divisors $\mathrm{K}_4, \dots, \mathrm{K}_{15}$ of  $\operatorname{NS}(\mathcal{Z})$ with $\mathrm{K}_n \circ  \mathrm{F}=0$ and $\mathrm{K}_m \circ \mathrm{K}_n = -2 \delta_{m n}$ for $4 \le m, n \le 15$.  We have the following:
\begin{corollary}
\label{cor:sections_and_divisors}
Under the equivalence in Proposition~\ref{prop:JAC}, the four sections $\{ p''_n \}_{n=1}^4$ from Lemma~\ref{lem:sections} are mapped to the zero-section $\sigma_\mathcal{Z}$ and three non-torsion sections $\{S_m\}_{m=1}^3$ of  $\pi_\mathcal{Z}$. The sections define divisor classes $\mathrm{K}_0 = [ \sigma_\mathcal{Z} ]$ and $\mathrm{K}_m = [S_m ]$ with $\mathrm{K}_m \circ \mathrm{F} =1$ and $\mathrm{K}_m \circ \mathrm{K}_n = -2 \delta_{m n}$ for $0 \le m \le 3$ and $0 \le n \le 15$ where $\mathrm{K}_4, \dots, \mathrm{K}_{15}$ are the non-neutral components of the reducible fibers of type $A_1$. In particular, the Jacobian elliptic fibration is never singular along $\sigma_\mathcal{Z}, S_1, S_2, S_3$.
\end{corollary}
\begin{proof}
By a direct computation one shows that the sections $S_1, S_2, S_3$ do not intersect each other, nor $\sigma_\mathcal{Z}$, nor any non-neutral components of the reducible fibers.
\end{proof}
\par The 2-torsion sections $T_1=\tau_\mathcal{Z}, T_2, T_3$, spanning $\operatorname{MW}(\mathcal{Z}, \pi_\mathcal{Z})_{\mathrm{tor}}$, each intersect the non-neutral components of eight reducible fibers of type $A_1$  -- partitioning the twelve rational curves (of the non-neutral components) into three sets of eight curves with pairwise intersections consisting of four curves and no triple intersection. The 2-torsion sections  do not intersect the zero section, but each 2-torsion section intersects each of the sections $S_1, S_2, S_3$ twice. The intersection pairings for all aforementioned divisor classes and height pairings of the corresponding sections are given in Table~\ref{tab:intersection}. 
\begin{table}
\parbox{.45\linewidth}{
\scalemath{0.8}{
\begin{tabular}{c||r|r|r|r|r|r|r|r}
$\circ$ 	& $F$ 	& $\sigma_\mathcal{Z}$	& $T_1$	& $T_2$	& $T_3$	& $S_1$	& $S_2$ 	& $S_3$ 	\\
\hline\hline
$F$					& 0		& $1$	& $1$	& $1$	& $1$	& $1$	& $1$	& $1$	\\
$\sigma_\mathcal{Z}$	& $1$	& $-2$	& $0$	& $0$	& $0$	& $0$	& $0$	& $0$	\\
$T_1$				& $1$	& $0$	& $-2$	& $0$	& $0$	& $2$	& $2$	& $2$	\\
$T_2$				& $1$	& $0$	& $0$	& $-2$	& $0$	& $2$	& $2$	& $2$	\\
$T_3$				& $1$	& $0$	& $0$	& $0$	& $-2$	& $2$	& $2$	& $2$	\\
$S_1$				& $1$	& $0$	& $2$	& $2$	& $2$	& $-2$	& $0$	& $0$	\\
$S_2$				& $1$	& $0$	& $2$	& $2$	& $2$	& $0$	& $-2$	& $0$	\\
$S_3$				& $1$	& $0$	& $2$	& $2$	& $2$	& $0$	& $0$	& $-2$	\\
\end{tabular}}}
\quad
\parbox{.45\linewidth}{
\scalemath{0.82}{
\begin{tabular}{c||r|r|r|r|r|r|r}
$\langle\bullet,\bullet\rangle$ 	& $\sigma_\mathcal{Z}$	& $T_1$	& $T_2$	& $T_3$	& $S_1$	& $S_2$ 	& $S_3$	\\
\hline\hline
$\sigma_\mathcal{Z}$	& $0$	& $0$ 	& $0$	& $0$ 	& $0$	& $0$	& $0$\\
$T_1$	& $0$	& $0$ 	& $0$	& $0$	& $0$	& $0$	& $0$\\
$T_2$	& $0$	& $0$ 	& $0$	& $0$ 	& $0$	& $0$	& $0$\\
$T_3$	& $0$	& $0$ 	& $0$	& $0$ 	& $0$	& $0$	& $0$\\
$S_1$	& $0$	& $0$ 	& $0$	& $0$	& $4$	& $2$	& $2$\\
$S_2$	& $0$	& $0$ 	& $0$	& $0$ 	& $2$	& $4$	& $2$\\
$S_3$	& $0$	& $0$ 	& $0$	& $0$ 	& $2$	& $2$	& $4$\\
\end{tabular}}}
\medskip
\caption{Intersection and Height Pairings}
\label{tab:height}\label{tab:intersection}
\end{table} 
We make the following:
\begin{remark}
\label{rem:Zfibration_dual}
A second Jacobian elliptic K3 surface $\pi_{\mathcal{Z}'}: \mathcal{Z}' \to \mathbb{P}^1$ is given by the Weierstrass model
\beq
\label{eqn:JacFib_dual}
 y^2 = x^3 +4 \Big(\mu \, P(x_0) + \nu \, Q(x_0)  \Big) \, x^2 + 4 \Big(\mu^2-\kappa \nu\Big) P(x_0)^2 x \,,
\eeq
with zero section $\sigma_{\mathcal{Z}'}$ and the 2-torsion section $\tau_{\mathcal{Z}'}: (x,y)=(0,0)$.  Generically, the model has four singular fibers of Kodaira type $I_4$, eight singular fibers of type $I_1$, and a Mordell-Weil group with $\operatorname{MW}(\mathcal{Z}', \pi_{\mathcal{Z}'})_{\mathrm{tor}}=\mathbb{Z}/2\mathbb{Z}$ and $\operatorname{rank} \operatorname{MW}(\mathcal{Z}', \pi_{\mathcal{Z}'}) = 3$. 
\par The K3 surfaces $\mathcal{Z}$ in Proposition~\ref{prop:JAC} and $\mathcal{Z}'$ are related by a pair of dual geometric 2-isogenies similar to Equation~(\ref{isog_intro}), i.e.,
\Beq
\label{isog_intro2}
 \xymatrix 
{ \mathcal{Z} \ar @(dl,ul) ^{\imath_{\mathcal{Z}}}
\ar @/_0.5pc/ @{-->}  [rr] &
& \mathcal{Z}' \ar @(dr,ur) _{\imath_{\mathcal{Z}'}}
\ar @/_0.5pc/ @{-->}  [ll] \\
} 
\Eeq
The isogenies are covered by the van~Geemen-Sarti involutions $\imath_\mathcal{Z}$ and $\imath_{\mathcal{Z}'}$ obtained as translations by the 2-torsion section $\tau_\mathcal{Z}: (X,Y)=(0,0)$ on $\mathcal{Z}$ and $\tau_{\mathcal{Z}'}: (x,y)=(0,0)$ on $\mathcal{Z}'$, respectively.
\end{remark}
 \subsection{Canonical curves of genus three}
 We will now construct a family of plane, quartic curves $\mathcal{D}_{x_0}$ with a bielliptic involution $\jmath$ over the projective line $\mathbb{P}_{x_0}$. Generically $\mathcal{D}_{x_0}$ is a smooth curve of genus three. We have the following:
\begin{proposition}
\label{prop:pencil}
Assuming Equations~(\ref{eqn:constraint0}) and~(\ref{eqn:constraintsEC}), the family $\{ \mathcal{D}_{x_0} \}_{x_0 \in \mathbb{P}_{x_0}}$ over the projective line $\mathbb{P}_{x_0}$ given by
\beq
\label{eqn:genus-three}
 \mathcal{D}_{x_0}: \quad  \, z^4 + 2 \, B(x, x_0) \, z^2 + 4 \, \big(\gamma-\delta\big)^2 \, P(x_0) \, P(x) =0 \,,
\eeq
is a linear pencil of plane, quartic curves with affine coordinates $x, z$, $B(x, x_0)$ given in Equation~(\ref{eqn:def_B}), and bielliptic involution $\jmath: (x, z) \mapsto (x,-z)$ covering  the degree-two map $\pi_{x_0}:  \mathcal{D}_{x_0} \to \mathcal{Q}_{x_0}= \mathcal{D}_{x_0}  / \langle \jmath \rangle$ such that
\begin{enumerate}
\item $\mathcal{D}_{x_0}$ is a smooth curve of genus three if and only if $\Delta_\mathcal{Z} \not =0$ in Equation~(\ref{eqn:discriminant}),
\item the pencil induces the elliptic fibration $\mathcal{Z}= \coprod_{x_0} \mathcal{Q}_{x_0} \to \mathbb{P}_{x_0}$ in Equation~(\ref{eqn:genus-one}),
\item the branch locus of $\psi$ is given by the divisors $\mathrm{K}_0, \mathrm{K}_1, \mathrm{K}_2, \mathrm{K}_3$ or, equivalently, $p''_n$ in Corollary~\ref{cor:sections_and_divisors}.
\end{enumerate}
\end{proposition}
\begin{proof}
For $P(x_0)=0$ the curve $\mathcal{D}_{x_0}$ is reducible: it consists of a rational component $z=0$ of multiplicity two and the conic $z^2= -2 B(x, x_0)$. The latter is irreducible since the discriminant $ \operatorname{Discr}_x B(x, x_0)$ does not vanish at a root of $P(x_0)=0$. Now assume $P_0 = P(x_0)\not =0$: Equation~(\ref{eqn:genus-three}) cannot have a singularity for $z=0$ since $\operatorname{Discr}_x(P) \not =0$ whence $z \not =0$. Taking the derivative of Equation~(\ref{eqn:genus-three}) with respect to $z$ at a singular point $(x, z)$ yields $z^2=-B(x)$, and $B(x)^2 - 4 P_0 P(x)=0$ from Equation~(\ref{eqn:genus-three}). The vanishing of the derivative of Equation~(\ref{eqn:genus-three}) with respect to $x$ yields $2 \, B(x) \, B'(x) - 4 (\gamma-\delta)^2 P_0 \, P'(x)=0$. Thus, for $\mathcal{D}_{x_0}$ to be reducible or to have a singular point, we must have 
\beq
  0  = \operatorname{Discr}_x\Big( B(x)^2 - 4 \, (\gamma-\delta)^2 P_0 P(x) \Big) = \Delta_\mathcal{Z} \,.
\eeq  
This proves (1).  We obtain a double cover $\pi_{x_0} \mathcal{D}_{x_0} \to \mathcal{Q}_{x_0}$ by setting $w=z^2 +B(x, x_0)$ in Equation~(\ref{eqn:genus-one}) giving Equation~(\ref{eqn:genus-three}).  It follows that $\mathcal{D}_{x_0}$ is the double cover of the curve $\mathcal{Q}_{x_0}$ branched over four points, and the two sheets of the covering are interchanged by the involution $\jmath: (x, z) \mapsto (x,-z)$. Thus, the pencil induces the elliptic fibration $\mathcal{Z}= \coprod_{x_0} \mathcal{Q}_{x_0} \to \mathbb{P}_{x_0}$ in Equation~(\ref{eqn:genus-one}) by means of the quotient $\mathcal{Q}_{x_0} \cong \mathcal{D}_{x_0}/ \langle \jmath \rangle$. This proves (2). Finally, (3) follows from Lemma~\ref{lem:sections} and Corollary~\ref{cor:sections_and_divisors}.
\end{proof}
\par We have the following:
\begin{proposition}
\label{prop:decomp}
For any smooth curve $\mathcal{D}_{x_0}$ in Equation~(\ref{eqn:genus-three}) with bielliptic structure the map $\pi_{x_0}: \mathcal{D}_{x_0} \to \mathcal{Q}_{x_0}$ induces an isogeny
\beq
 \operatorname{Jac}{(\mathcal{D}_{x_0})} \ \simeq \ \operatorname{Prym}{(\mathcal{D}_{x_0}, \pi_{x_0} )} \times \operatorname{Jac}(\mathcal{Q}_{x_0}) \,,
\eeq 
where $\operatorname{Prym}{(\mathcal{D}_{x_0}, \pi_{x_0})}$ is the Prym variety with a polarization of type $(1,2)$. In particular, $\mathcal{D}_{x_0}$ is embedded into  $\operatorname{Prym}{(\mathcal{D}_{x_0}, \pi_{x_0})}$ as a curve of self-intersection four.
\end{proposition}
\begin{proof}
Assuming Equations~(\ref{eqn:constraint0}) and~(\ref{eqn:constraintsEC}), $\mathcal{D}_{x_0}$ in Equation~(\ref{eqn:genus-three}) is smooth, bielliptic with genus three if and only if $\Delta_\mathcal{Z} \not =0$ in Equation~(\ref{eqn:discriminant}).  $\mathcal{D}_{x_0}$ is the double cover of the curve $\mathcal{Q}_{x_0}$ branched over four points. The double covering~$\psi$ induces an associated norm morphism $\operatorname{Jac}{(\mathcal{D}_{x_0})} \to \operatorname{Jac}(\mathcal{Q}_{x_0})$. The involution $\imath$ extends to an involution on $\operatorname{Jac}{(\mathcal{D}_{x_0})} \to \operatorname{Jac}(\mathcal{Q}_{x_0})$. Then $\operatorname{Jac}{(\mathcal{D}_{x_0})}$ splits into an even part and an odd part. By definition the latter is the Prym variety. It follows from \cite{MR946234}*{Sec.~1.4}  that the Prym has a natural polarization on it, induced by the theta divisor on $\operatorname{Jac}{(\mathcal{Q}_{x_0})}$, which is the theta divisor $\{ [p - p_6] | \, p \in \operatorname{Jac}{(\mathcal{Q}_{x_0})} \}$ where $p_6$ defines the neutral element of the elliptic group law such that
\beq
 \operatorname{Prym}{(\mathcal{D}_{x_0}, \pi_{x_0})} \cong \operatorname{Jac}{(\mathcal{D}_{x_0})} / \pi_{x_0}^*\operatorname{Jac}(\mathcal{Q}_{x_0}) \,.
\eeq 
Barth also proved that a smooth, bielliptic curve of genus three is embedded into  $\operatorname{Prym}{(\mathcal{D}_{x_0}, \pi_{x_0})}$ as a curve of self-intersection four.
\end{proof}
\begin{remark}
The notion of Prym variety in Proposition~\ref{prop:decomp} can be generalized to include the singular covers $\pi_{x_0}: \mathcal{D}_{x_0} \to \mathcal{Q}_{x_0}$ using the results of \cite{MR572974}*{Prop.~3.5} and \cite{MR1736231}*{Lemma~1}. The Prym is then replaced by a generalized Prym variety for an allowable cover birational to the singular cover $\pi_{x_0}: \mathcal{D}_{x_0} \to \mathcal{Q}_{x_0}$.
\end{remark}
 \subsection{Singular and hyperelliptic fibers}
In this section, we shall examine the singular and hyperelliptic elements of the pencil of curves of genus three after some elementary modifications.  Substituting $w=y + B(x, x_0)$ into Equation~(\ref{eqn:genus-one}) we obtain its equivalent form
\beq
 \mathcal{Q}_{x_0}: \quad  \, y^2 \ + 2 \, B(x, x_0) \, y \ + 4 \, (\gamma-\delta)^2  \, P(x_0) \, P(x) =0 \,.
\eeq
The double cover $\pi_{x_0}: \mathcal{D}_{x_0} \to  \mathcal{Q}_{x_0}$ is then simply given by $y =z^2$, and the four branch points are $(x,y)=(x''_n,0)$ with $P(x''_n)=0$ for $1 \le n \le 4$. Blowing up at the points $P(x_0)=0$, we set $y = 2 \, P(x_0) \, \tilde{y}$ and take the strict transform to obtain 
\beq
 \widetilde{\mathcal{Q}}_{x_0}: \quad  \, P(x_0) \, \tilde{y}^2 \ + B(x, x_0) \, \tilde{y} \ + (\gamma-\delta)^2 \, P(x) =0 \,,
\eeq
and a double cover $\pi_{x_0}: \widetilde{\mathcal{D}}_{x_0} \to  \widetilde{\mathcal{Q}}_{x_0}$ given by $\tilde{y} = \tilde{z}^2$ with
\beq
\label{eqn:genus-three_blow-up}
 \widetilde{\mathcal{D}}_{x_0}: \quad  P(x_0) \,  \tilde{z}^4 \ + B(x, x_0) \, \tilde{z}^2 \ + (\gamma-\delta)^2 P(x) =0 \,,
\eeq
and the bielliptic involution $\jmath: (x,\tilde{z}) \mapsto (x,-\tilde{z})$.  We have the following:
\begin{proposition}
\label{prop:normalization}
For $\Delta_\mathcal{Z} \not =0$ in Equation~(\ref{eqn:discriminant}) the curve $\widetilde{\mathcal{D}}_{x_0}$ is a smooth irreducible curve of genus three isomorphic to $\mathcal{D}_{x_0}$. For $\Delta_\mathcal{Z} =0$ there are twelve singular curves forming three sets of four isomorphic curves over the roots of $\kappa P(x_0) + (\mu \pm \sqrt{\mu^2 - \kappa \nu}) Q(x_0)=0$ and $P(x_0)=0$, respectively. The former eight are irreducible curves of geometric genus two with one node. The latter four are reducible nodal curves isomorphic to $\mathbb{P}^1 \cup \mathcal{C}'$ where $\mathcal{C}'$ is the curve of genus two given by
\beq
\label{eqn:normalization} 
  \mathcal{C}': \quad  \eta^2 =  \big(\xi-\gamma\big)\big(\xi-\delta\big) \, S(\xi)\,.
\eeq
Here, $(\xi,\eta) \in \mathbb{C}^2$ are affine coordinates and $S(\xi)$ is given in Equation~(\ref{eqn:EC}).
\end{proposition}
\begin{proof}
One checks that the general element $\mathcal{D}_{x_0}$ is smooth and irreducible. It is bielliptic and of genus three by construction. For $P(x_0) \not = 0$ we obviously have $\widetilde{\mathcal{D}}_{x_0} \cong \mathcal{D}_{x_0}$ and $\widetilde{\mathcal{Q}}_{x_0} \cong \mathcal{Q}_{x_0}$.  Equation~(\ref{eqn:discriminant}) shows that there are twelve singular curves and one checks by an explicit coordinate transformation that the singular curves form three sets of four isomorphic curves. A curve over a root of $\kappa P(x_0) + (\mu \pm \sqrt{\mu^2 - \kappa \nu}) Q(x_0)=0$ is an irreducible curve of geometric genus two with one double point, which is easily seen to be a node.
\par  Let the polynomial $P(x)$ be given by $P(x)=\prod_n (x-x''_n)$. Setting $x_0=x''_n$ in Equation~(\ref{eqn:genus-three_blow-up}) and rescaling $\tilde{z} = i Z/(3 (\gamma-\delta) B(x, x_0))$ yields 
\beq
\label{eqn:normalization0}
 Z^2 =  \big(\gamma- \delta\big) \, P(x) \, \Big( 3 \, (x-x''_n) \gamma + \alpha \, x + \beta \Big)   \Big( 3 \, (x-x''_n) \delta + \alpha \, x + \beta \Big) \,,
\eeq 
where $\alpha$ and $\beta$ are cubic and quadratic polynomials in the coefficients $x''_1, \dots, x''_4$, respectively, symmetric in $\{ x''_m \}_{m\not = n}$. Equation~(\ref{eqn:normalization}) obviously defines a curve of genus two. We compute its Igusa-Clebsch invariants, using the same normalization as in \cites{MR3712162,MR3731039}.  Denoting the Igusa-Clebsch invariants of the curve of genus two in Equation~(\ref{eqn:normalization0}) and Equation~(\ref{eqn:normalization}) by $[ I_2 : I_4 : I_6 : I_{10} ] \in \mathbb{P}(2,4,6,10)$ and $[ I'_2 : I'_4 : I'_6 : I'_{10} ]$, respectively, one checks that
\beq
[ I_2 : I_4 : I_6 : I_{10} ] = [ r^2 I'_2 \ : \ r^4I'_4 \ : \ r^6I'_6 \ : \ r^{10}I'_{10} ] = [ I'_2 : I'_4 : I'_6 : I'_{10} ] \,,
\eeq
with $r=9(\gamma-\delta)(x''_1-x''_2)(x''_1-x''_3)(x''_1-x''_4)$. Thus, the two curves of genus two are isomorphic.
\end{proof}
We also have the following:
\begin{proposition}
\label{prop:hyperelliptic}
The curve $\widetilde{\mathcal{D}}_{x_0}$ in Equation~(\ref{eqn:genus-three_blow-up}) admits a hyperelliptic involution if and only if $x_0$ is a root of $[ P, Q ]_{x_0}=0$. There are six such hyperelliptic elements, and they are smooth if and only if $2(q_\gamma \pm q_\delta) \not = 0$ in Equation~(\ref{eqn:constraintsEC}).
\end{proposition}
\begin{proof}
If $\widetilde{\mathcal{D}}_{x_0}$ admits a hyperelliptic involution $k$ it commutes with the bielliptic involution $\jmath$. As the two involutions commute, $k$ induces a permutation on the fixed points of $\jmath$. For $P(x)=\prod_n (x-x''_n)$ as in the proof of Proposition~\ref{prop:normalization}, we define the fractional linear map $T$ given by
\beq
 x \mapsto T(x) = \frac{-(ab-cd)x+( a b c+a b d-a c d-b c d)}{-(a+b-c-d)x+(ab-cd)} \,,
\eeq
such that $x=T(T(x))$, $T(a)=b$, $T(b)=a$, $T(c)=d$, $T(d)=c$. We then set $a=x''_1$, $b=x''_2$, $c=x''_3$, $d=x''_4$ such that
\beq
 P\Big( T(x) \Big)  = \underbrace{\frac{(x''_1-x''_3)^2(x''_1-x''_4)^2(x''_2-x''_3)^2(x''_2-x''_4)^2}{\big( (x''_1+x''_2-x''_3-x''_4) x -(x''_1x''_2-x''_3x''_4)\big)^4}}_{=: \, C(x)^4} \, P(x) \,.
\eeq
It follows that if $\widetilde{\mathcal{D}}_{x_0}$ admits a hyperelliptic involution, it is of the form
\beq
  k: \quad \Big(x, z \Big) \mapsto \left( \frac{-(ab-cd)x+(a b c+a b d-a c d-b c d)}{-(a+b-c-d)x+(ab-cd)}, z \right) \,.
\eeq  
The rational functions $B\big(T(x)\big) - C(x)^2 \, B(x)$ and $P\big(T(x)\big) - C(x)^4 \, P(x)$ have the only common factor 
\beq
\label{eqn:r_commutator}
 r_{1234}(x_0) = \big(x''_1 + x''_2 - x''_3 - x''_4\big) x_0^2 - 2 \big( x''_1 x''_2 - x''_3 x''_4 \big) x_0 + x''_1 x''_2 \big( x''_3+x''_4\big) - \big(x''_1 + x''_2\big) x''_3 x''_4 \,.
\eeq
One then checks that a permutation of the roots yields
\beq
 r_{1234}(x_0) \, r_{1324}(x_0) \, r_{1423}(x_0) = - 4 \, [ P, Q ]_{x_0} \,. 
\eeq
Moreover, $[ P, Q ]_{x}$ is a polynomial of degree six. 
\par It easily follows that $P(x_0)$ and $[ P, Q ]_{x_0}$ never vanish at the same time, given our assumption in Equation~(\ref{eqn:constraint0}).  One then checks that for $\alpha, \beta \in \mathbb{C}$  with $\beta \not =0$, the resultant satisfies
\beq
\label{eqn:resultant}
 \operatorname{Res}_x\Big( \alpha P + \beta Q, [ P, Q ] \Big) = 2^{-8}  \operatorname{Discr}_x(P)^3  \, \beta^6  S\left(\frac{\alpha}{\beta}\right)^2\,.
\eeq 
Thus,  $\alpha P(x_0) + \beta Q(x_0)$ and $[ P, Q ]_{x_0}$ vanish simultaneously, if and only if $S(\alpha/\beta)=0$ where $S(\xi)$ was given in Equation~(\ref{eqn:EC}).
\par We use the elliptic group law on $\mathcal{E}$ to compute the coordinates of the points $\pm (q_\gamma + q_\delta),  \pm (q_\gamma - q_\delta) \in \mathcal{E}$ with coordinates $(\xi_{q_\gamma+q_\delta}, \, \pm \eta_{q_\gamma+q_\delta})$ and 
$(\xi_{q_\gamma-q_\delta}, \, \pm \eta_{q_\gamma-q_\delta})$, respectively. It follows that
\beq
 \Big\lbrace \xi_{q_\gamma+q_\delta}, \ \xi_{q_\gamma-q_\delta} \Big\rbrace  =  \left\lbrace \frac{\mu + \sqrt{\mu^2 - \kappa \nu}}{\nu},  \frac{\mu - \sqrt{\mu^2 - \kappa \nu}}{\nu} \right\rbrace\,.
\eeq 
The second factor of the discriminant in Equation~(\ref{eqn:discriminant}) is
\beq
\begin{split}
 & \quad \qquad \qquad \kappa \, P(x_0)^2 + 2 \mu \, P(x_0) Q(x_0) + \nu \, Q(x_0)^2 \\
 = & \, \nu \left( \frac{\mu + \sqrt{\mu^2 - \kappa \nu}}{\nu} P(x_0) + Q(x_0) \right)  \left( \frac{\mu - \sqrt{\mu^2 - \kappa \nu}}{\nu} P(x_0) + Q(x_0) \right) \,.
\end{split}
\eeq
Using Equation~(\ref{eqn:resultant}) it follows that $[ P, Q ]_{x_0}$ and the discriminant $\Delta_{\mathcal{Q}_{x_0}}$ do not have a common factor if and only if $S(\xi_{q_\gamma+q_\delta}) \not = 0$ and $S(\xi_{q_\gamma+q_\delta}) \not = 0$. This is equivalent to the points $q_\gamma \pm q_\delta$ \emph{not} being 2-torsion points of $\mathcal{E}$.
\end{proof}
 \subsection{Canonical curves of genus five}
 \label{sec:g5_curves}
We identify the smooth curve of genus three $\mathcal{D}=\widetilde{\mathcal{D}}_{x_0}$ in Equation~(\ref{eqn:genus-three_blow-up}) with its canonical model in the plane $\mathbb{P}^2=\mathbb{P}(X,Y,Z)$ given by $\mathbb{P}^2 \cong |\mathcal{K}_{\mathcal{D}}|^*$, and write
\beq
\label{eqn:g3curve_AG}
 \mathcal{D}: \quad a_0 \, Z^4 +  b_2(X,Y) \, Z^2 + c_4(X,Y) = 0 \,,
\eeq
such that $a_0 = P(x_0)$, $b_2(x,1) = B(x, x_0)$, and $c_4(x,1)= (\gamma-\delta)^2 P(x)$. We will also assume $a_0=P(x_0) \not =0$. 
\par We set $\mathcal{Q}=\widetilde{\mathcal{Q}}_{x_0}$ and the bielliptic involution is $\jmath: [X:Y:Z] \mapsto [X:Y:-Z]$ covers $\pi:  \mathcal{D} \to  \mathcal{Q}$ with branch locus $\mathrm{B} = p''_1 + p''_2 + p''_3 + p''_4 \subset \mathcal{Q}$ given in Lemma~\ref{lem:sections} which is an effective divisor of degree four without multiple points. Let $\mathscr{N}$ be the line bundle corresponding to half of the divisor class of $\mathrm{B}$, then $\mathscr{N}^2=\mathcal{O}_{\mathcal{Q}}(\mathrm{B})$.  Conversely, the data of $(\mathcal{Q}, \mathscr{N}, \mathrm{B})$ determines the double cover $\mathcal{D}$ uniquely up to isomorphism. By slight abuse of notation, we set $\{ p''_1, \dots, p''_4\} = \pi^{-1}(\mathrm{B}) \subset \mathcal{D}$ with $p''_n: [X:Y:Z]=[x''_n:1:0]$ and $P(x''_n)=0$ or, equivalently, $c_4(x''_n,1)=0$ for $1\le n \le4$.
\par The adjunction formula implies that the linear systems $|\mathcal{K}_{\mathcal{Q}}|^*$ and $|\mathscr{N}|^*$ can be identified in the projective plane $|\mathcal{K}_{\mathcal{D}}|^*$ with a point $\mathrm{O}$ and a line $\mathrm{L}_0$, respectively \cite{MR1816214}. There is a classical characterization of the data $(\mathrm{O}, \mathrm{L}_0)$: it is well known that $\jmath$ on a canonical curve of genus three is induced by a projective involution $\tilde{\jmath}$ whose set of fixed points consists of a point $\mathrm{O}$ and a line $\mathrm{L}_0$ such that the intersection $\mathrm{L}_0 \cap \mathcal{D}$ are the fixed points of $\jmath$. Since the points $p''_n: [X:Y:Z]=[x''_n:1:0]$ in Equation~(\ref{eqn:g3curve_AG}) with $c(x''_n,1)=0$ for $1\le n \le4$ are the ramification points of $\pi: \mathcal{D} \to \mathcal{Q}$, we obtain $\mathrm{L}_0 = \mathrm{V}(Z)$ and $\mathcal{D} \cap \mathrm{L}_0 = p''_1 + \dots + p''_4$. The tangent lines at the points $p''_n$ are $\mathrm{V}(X - x''_n \, Y)$ for $1 \le n \le 4$, and they all must pass through the point $\mathrm{O}$ \cite{MR1816214}*{Thm.~2.5} whence $\mathrm{O}: [X:Y:Z]=[0:0:1]$. 
\par On the other hand, $\chi: \mathcal{Q} \to \mathbb{P}^1$ has the ramification divisor $p'_1 + p'_2 + p'_3 + p'_4 \subset \mathcal{Q}$; see Section~\ref{ssec:AJ_K3}. The preimages of $p'_n$ in $\mathcal{D}$ are pairs of points $p'_{n ,\pm}: [X:Y:Z]=[x'_n:1: \pm z'_n]$ with $4 \, a_0 c_4(x'_n,1)-b_2(x'_n,1)^2=0$ and $2 a_0 (z'_n)^2 + b_2(x'_n,1) =0$ for $1\le n \le4$. The tangent line $\mathrm{L}_n$ at $p'_{n ,\pm}$ is given by $\mathrm{L}_n= \mathrm{V}(X - x'_n Y)$, and all $\mathrm{L}_n$ pass through the same point $\mathrm{O}: [X:Y:Z]=[0:0:1]$. The lines $\mathrm{L}_n$ are in fact bitangents with intersection divisors $\mathcal{D} \cap \mathrm{L}_n = 2 p'_{n, +} + 2 p'_{n, -}$. This characterization of the bielliptic structure in terms of bitangents is originally due to Kovalevskaya; see Dolgachev \cite{dolgachev2014endomorphisms} and work by the authors \cite{CMS:2019}:
\begin{theorem}[Kovalevskaya] 
\label{thm:Kowalevskaya}
The point $\mathrm{O}$ is the intersection point of four distinct bitangents $\mathrm{L}_n$ of $\mathcal{D}$ with $1 \le n \le 4$.  Conversely, if a plane quartic has four bitangents $\mathrm{L}_n$ intersecting at a point $\mathrm{O}$, then there exists a bielliptic involution $\imath$ of $\mathcal{D}$ such that the projective involution $\tilde{\imath}$ has $\mathrm{O}$ as its isolated fixed point.
\end{theorem}
\par It is well known that a smooth plane quartic has exactly 28 bitangents; together with the points of order two on $\operatorname{Jac}(\mathcal{D})$ they have a rich symmetry, called the \emph{$64_{28}$-symmetry}.  The established normal form for $\widetilde{\mathcal{D}}_{x_0}$ in Equation~(\ref{eqn:g3curve_AG}) determines a grouping of four bielliptic tangents into two pairs as follows: because of $\mathcal{D} \cap \mathrm{L}_n = 2 p'_{n, +} + 2 p'_{n, -}$, each divisor $2 p'_{n, +} + 2 p'_{n, -}$ for $1 \le n \le 4$ is a canonical divisor, and $\Theta_n= p'_{n, +} + p'_{n, -}$ is a theta divisor, i.e., a point in $\operatorname{Pic}^2(\mathcal{D})$ such that $2\, \Theta_n \sim  \mathcal{K}_{\mathcal{D}}$. The difference of any pair of theta divisors $\Theta_n - \Theta_m$ is a point of order two in $\operatorname{Jac}(\mathcal{D})$. Since $\sum_n \Theta_n \sim  2K_\mathcal{D}$ there exists a conic $\mathrm{C}_0$ that cuts out the divisor $\sum_n p'_{n, +} + p'_{n, -}$ on $\mathcal{D}$. Since $2 \, \mathrm{C}_0$ and $\mathrm{L}_1+ \dots + \mathrm{L}_4$ cut out the same divisor on $\mathcal{D}$, the equation for $\mathcal{D}$ can be re-written as
\beq
  \mathcal{D}: \quad  q_0^2 = l_1 \cdots l_4 \,,
\eeq  
where $\mathrm{L}_n = \mathrm{V}(l_n)$ and $\mathrm{C}_0 = \mathrm{V}(q_0)$ with $q_0(X,Y,Z) = 2 a_0 Z^2 +b_2(X,Y)$. Notice that because of $\Theta_n - \Theta_m \sim  \Theta_n + \Theta_m - K_\mathcal{D}$, the differences of any two pairs of theta divisors always add up to zero, i.e., for $\{ m, n, r, s \} = \{ 1, 2, 3, 4\}$ we have
\beq
\label{eqn:sum_divisors}
 \Big( \Theta_m - \Theta_n \Big) +  \Big( \Theta_r - \Theta_s \Big)  \sim \sum^4_{n=1} \Theta_n - 2 \, K_\mathcal{D} =0    \,.
\eeq 
On the other hand, grouping the four bitangents from Theorem~\ref{thm:Kowalevskaya} into two pairs or, equivalently, the choice of $\pm (\Theta_m - \Theta_n) \in \operatorname{Jac}(\mathcal{D})[2]$ (or $\pm (\Theta_r - \Theta_s)$), amounts to combining pairs of lines into two quadrics $q_1 = l_m l_n$ and $q_2 =l_r l_s$ and writing the bielliptic curve $\mathcal{D}$ as plane projective model
\beq
\label{eqn:normal_form}
 q_0\big(X,Y, Z\big)^2 = q_1\big(X,Y \big) \,  q_2\big(X,Y \big)  \,.
\eeq 
Our construction in Section~\ref{ssec:AJ_K3} naturally provides such a grouping of bitangents into two pairs for $\mathcal{D}=\widetilde{\mathcal{D}}_{x_0}$. In fact, for the normal form given in Equation~(\ref{eqn:g3curve_AG}) the three conics $\mathrm{C}_i = \mathrm{V}(q_i)$ with $0\le i \le 2$ are given by
\beq
\label{eqn:g5_quadrics}
\begin{split}
 q_1(x,1) & = \gamma^2  (x - x_0)^2  - 4 \, \gamma \, R(x, x_0)  - 4 R_1(x, x_0) \,\\
 q_2(x,1) & = \delta^2  (x - x_0)^2  - 4 \, \delta \, R(x, x_0)  - 4 R_1(x, x_0)  \,\\
 q_0(x,1,\tilde{z}) & = 2 \, P(x_0)\, \tilde{z}^2 +  \gamma \delta  \,  (x - x_0)^2  - 2 (\gamma+\delta) \, R(x, x_0) - 4 R_1(x, x_0)  \,.
\end{split}
\eeq
\par It was proven in \cite{MR2406115} that the curves of genus three of the form~(\ref{eqn:normal_form}) admit an unramified double cover $\rho'_{x_0}: \mathcal{F}_{x_0} \to \widetilde{\mathcal{D}}_{x_0}$ where the double cover $\mathcal{F}_{x_0}$ is a non-hyperelliptic curve of genus five given as the intersection of the three quadrics $Q_i$ with $0\le i \le 2$ in $\mathbb{P}^4=\mathbb{P}(V,W,X,Y,Z)$ given by
\beq
\label{eqn:genus_5_curve}
\mathcal{F}_{x_0}: \quad \left\lbrace \begin{array}{lclcl} 0 &=& Q_0(V,W,X,Y,Z)& = & q_0(X,Y,Z) -VW  \\  0  &=&  Q_1(V,W,X,Y)& = & q_1(X,Y)  -V^2 \\   0 &=& Q_2(V,W,X,Y)& = & q_2(X,Y) - W^2\end{array} \right. .
\eeq
The involution
\beq
\label{eqn:genus_5_involution}
 \imath': \mathbb{P}^4 \to \mathbb{P}^4 \,, \quad [V: W: X: Y:Z ] \mapsto [-V: -W: X: Y:Z ] \,,
\eeq
interchanges the sheets of the double cover $\rho'_{x_0}: \mathcal{F}_{x_0} \to \widetilde{\mathcal{D}}_{x_0}$.  Conversely, the canonical model of any non-hyperelliptic, (non-trigonal) curve of genus five is the intersection of three quadrics in $\mathbb{P}^4$ by Petri’s Theorem~\cite{MR770932}*{p.~131}. We have the following:
\begin{lemma}
The involution $\imath$ has no fixed points iff $\Delta_\mathcal{Z} \!\!\mid_{x_0} \not =0$ in Equation~(\ref{eqn:discriminant}).
\end{lemma}
\begin{proof}
First assume $P(x_0)\not=0$: the quadrics $q_1$ and $q_2$ have a common zero if $\Delta_\mathcal{Z} \mid_{x_0} =0$. Since
\beq
  q_0(x,1,\tilde{z})  = 2 \, P(x_0)\, \tilde{z}^2  + \frac{1}{2} \Big(q_1(x,1)  + q_2(x,1) - (\gamma-\delta)^2 (x - x_0)^2 \Big) \,,
\eeq
we can then solve $q_0=0$ to find the fixed points of the involution. Next, we observe that the discriminants $\operatorname{Discr}_x(q_1)$ and $\operatorname{Discr}_x(q_2)$ and the resultant $\operatorname{Res}_x(R, x-x_0)$ are all proportional to $P(x_0)$. Using Equation~(\ref{eqn:relat}) we thus have a fixed locus for the involution for $P(x)=P(x_0)=0$.
\end{proof}
\begin{remark}
\label{rem:hrouping_bielliptic_tangents}
The constructed unramified double cover $\rho'_{x_0}: \mathcal{F}_{x_0} \to \widetilde{\mathcal{D}}_{x_0}$ corresponds to choosing one out of three possible groupings of the four marked bielliptic tangents into two pairs. Each choice is determined by an element $\mathrm{D}=\pm(\Theta_m - \Theta_n)\in \operatorname{Jac}(\mathcal{D})[2]$, or, equivalently, $\mathrm{D}=\pm(\Theta_k - \Theta_l)$ with $\{ k, l, m, n \} = \{ 1, 2, 3, 4\}$; see Equation~(\ref{eqn:sum_divisors}). In turn, $\mathrm{D}$ is a divisor of degree zero with associated line bundle $\mathscr{L}'=\mathcal{O}_\mathcal{D}(\mathrm{D})$ satisfying $\mathscr{L}^{\prime \, \otimes 2} = \mathcal{O}_\mathcal{D}$. The zero section of $\mathscr{L}'$ then determines the unramified double cover $\rho'_{x_0}: \mathcal{F}_{x_0} \to \widetilde{\mathcal{D}}_{x_0}$ uniquely (up to isomorphism).
\end{remark}
\par We also have the analogue of Lemma~\ref{lem:sections}:
\begin{lemma}
\label{lem:sections_g5}
On $\mathcal{F}_{x_0}$ there are eight points -- rational over $\mathbb{Q}(x_0)$ -- with $X=x''_n$, $Y=1$, $Z=0$ where $\{x''_n\}_{n=1}^4$ are the roots of $P(x)=0$.
\end{lemma}
\begin{proof}
The proof follows by checking that for $X=x''_n$, $Y=1$ the quadrics $q_1, q_2$  in Equation~(\ref{eqn:g5_quadrics}) are perfect squares with roots $\pm V$ and $\pm W$ such $VW=q_0(X,Y,0)$.
\end{proof}
\par Using the Riemann-Roch theorem, it follows that the hypernet $Q(\alpha_0,\alpha_1, \alpha_2)=\alpha_0 Q_0 + \alpha_1 Q_1 + \alpha_2 Q_2$ is precisely the linear system of all quadrics in $\mathbb{P}^4$ containing $\mathcal{F}_{x_0}$ in Equation~(\ref{eqn:genus_5_curve}). Let $\Gamma$ be the locus of quadrics of rank less or equal to four, i.e.,
\beq 
 \Gamma = \Big\{ [\alpha_0 : \alpha_1 : \alpha_2 ] \in \mathbb{P}^2 \ \Big| \  \det{Q(\alpha_0,\alpha_1, \alpha_2)} =0 \Big \} \,,
\eeq
where the quadrics $Q_i$ for $0\le i \le 2$ are identified with the symmetric five-by-five matrices corresponding to the quadratic forms they represent. A simple computation shows that $\Gamma= \Gamma^+ \cup \Gamma^-$ is one-dimensional with
\beq
  \Gamma^+  = \mathrm{V}\Big(  \det{\big(\alpha_0 \, q_0 + \alpha_1 \, q_1 + \alpha_2 \, q_2\big)} \Big) \,,\quad
  \Gamma^-  =  \mathrm{V}\left( \det{\left(\begin{smallmatrix} \alpha_1 & \frac{1}{2} \alpha_0 \\  \frac{1}{2} \alpha_0 &  \alpha_2 \end{smallmatrix}\right)}\right)= \mathrm{V}\Big(\alpha_0^2 - 4 \alpha_1 \alpha_2\Big) \,.
\eeq
Thus, $\Gamma$ consists of a cubic curve $\Gamma^+$ and a conic $\Gamma^-$ without multiple components. The singular locus of $\Gamma$ is the zero-dimensional locus $\Gamma' \subset \Gamma$ of quadrics of rank less or equal to three, and the singularities of $\Gamma$ are all ordinary nodes. We also consider the scheme of special divisors on $\mathcal{F}_{x_0}$, given by
\beq
 W^1_4(\mathcal{F}_{x_0}) = \Big \{ \mathrm{D} \in \operatorname{Pic}^4(\mathcal{F}_{x_0}) \; \Big| \; h^0(\mathcal{F}_{x_0}, \mathcal{O}_{\mathcal{F}_{x_0}}(D)) \ge 2 \Big\} \,,
\eeq
which is equipped with a natural map $\phi: W^1_4(\mathcal{F}_{x_0}) \to \Gamma$ of degree two branched exactly over $\Gamma'$ \cite{MR2406115}*{Cor.~4.2}. 
\par One irreducible component of $W^1_4(\mathcal{F}_{x_0})$ is $\mathcal{C}'=\phi^{-1}(\Gamma^-)$, i.e., the  double cover of $\mathbb{P}^1\cong \Gamma^-$ branched on the six points of $\Gamma^+ \cap \Gamma^-$. One can show that the Jacobian of $\mathcal{C}'$ spans a two-dimensional abelian sub-variety in $\operatorname{Jac}(\mathcal{F}_{x_0})$~\cite{MR770932}. In fact, using the rational parametrization $[\alpha_0 : \alpha_1 : \alpha_2 ] = [2\xi : 1: \xi^2]$ for $\Gamma^-$ we obtain an explicit equation for $\mathcal{C}'$. The following was proved in \cite{MR422289}, \cite{MR770932}*{Ex.~VI.F}  and in \cite{MR2406115} over a general field of characteristic zero:
\begin{proposition}
\label{prop:genus_5_prym}
In the situation above, we have 
\beq
  \operatorname{Prym}(\mathcal{F}_{x_0}, \rho'_{x_0}) = \operatorname{Jac}(\mathcal{C}') \,,
\eeq  
where the smooth curve $\mathcal{C}'$ of genus two is given by
\beq
\label{eqn:normalization2}
\mathcal{C}': \quad \eta^2 = - \det{\Big( 2 \, \xi \, q_0 +  q_1 + \xi^2 \, q_2 \Big)} \,,
\eeq
and the conics $q_i$ for $0\le i \le 2$ are considered symmetric three-by-three matrices corresponding to the quadratic forms they represent.
\end{proposition}
For the curves of genus five $\mathcal{F}_{x_0}$ over $\widetilde{\mathcal{D}}_{x_0}$ we have the following:
\begin{corollary}
\label{cor:genus_5_prym}
For $\Delta_\mathcal{Z} \not =0$ the curve $\mathcal{F}_{x_0}$ in Equation~(\ref{eqn:genus_5_curve}) is a smooth curve of genus five admitting the unramified double cover $\rho'_{x_0}: \mathcal{F}_{x_0} \to \widetilde{\mathcal{D}}_{x_0}$. The Prym variety is canonically isomorphic to the principally polarized abelian surface  $\operatorname{Prym}(\mathcal{F}_{x_0}, \rho'_{x_0}) = \operatorname{Jac}(\mathcal{C}')$ where $\mathcal{C}'$ is isomorphic to the smooth curve of genus two in Equation~(\ref{eqn:normalization}).
\end{corollary}
\begin{proof}
We compute the Igusa-Clebsch invariants, using the same normalization as in \cites{MR3712162,MR3731039}.  Denoting the Igusa-Clebsch invariants of the curves of genus two in Equation~(\ref{eqn:normalization}) and~(\ref{eqn:normalization2}) by $[ I_2 : I_4 : I_6 : I_{10} ] \in \mathbb{P}(2,4,6,10)$ and $[ I'_2 : I'_4 : I'_6 : I'_{10} ]$, respectively, one checks that
\beq
[ I_2 : I_4 : I_6 : I_{10} ] = [ r^2 I'_2 \ : \ r^4I'_4 \ : \ r^6I'_6 \ : \ r^{10}I'_{10} ] = [ I'_2 : I'_4 : I'_6 : I'_{10} ] \,,
\eeq
with $r=32 \, (\gamma-\delta) \, P(x_0)^2$. Thus, the curves are isomorphic.
\end{proof}
\par  Because of Lemma~\ref{lem:sections_g5} we can embed $\mathcal{F}_{x_0}$ into $\operatorname{Jac}(\mathcal{F}_{x_0})$. We then combine this map with the projection map $\mathrm{id}_* -\jmath_*: \operatorname{Jac}(\mathcal{F}_{x_0}) \to \operatorname{Prym}(\mathcal{F}_{x_0}, \rho'_{x_0})$, which is called the \emph{Abel-Prym map}. We have the following:
\begin{lemma}
\label{lem:APembedding}
Each smooth curve $\mathcal{F}_{x_0}$ embeds into $\operatorname{Jac}(\mathcal{C}')$ via the Abel-Prym map.
\end{lemma}
\begin{proof}
The curve $\mathcal{F}_{x_0}$ is embeds into $\operatorname{Jac}(\mathcal{F}_{x_0})$ which decomposes into $\operatorname{Prym}{(\mathcal{F}_{x_0}, \rho'_{x_0})}$ and $\operatorname{Jac}(\widetilde{\mathcal{D}}_{x_0})$. It was proved in \cite{MR422289}*{Prop.~5.3} that for an unramified double cover $\rho'_{x_0}: \mathcal{F}_{x_0} \to \widetilde{\mathcal{D}}_{x_0}$ this embedding misses $\operatorname{Jac}(\widetilde{\mathcal{D}}_{x_0})$. Verra proves that the curves of genus five $\mathcal{F}_{x_0} \subset \operatorname{Jac}(\mathcal{C}')$ of the form in Equation~(\ref{eqn:genus_5_curve}), up to translation by a 2-torsion point, are Abel-Prym embeddings \cite{MR875339}.
\end{proof}
\par In Equation~(\ref{eqn:genus_5_curve})  $q_1, q_2$ do not depend on the variable $Z$; thus, $\Gamma^+ = \mathrm{L} \cup  \mathrm{C}$ decomposes into a line component $\mathrm{L}$ and another irreducible conic $\mathrm{C}$. In general, there is a bijection between lines in $\Gamma$ and bielliptic structures on $\mathcal{F}$. In fact, the following was proved in \cite{MR770932}*{Ex.~VI.F}:
\begin{lemma}
If $\mathrm{L}\subset \Gamma$ is a line component and $\mathcal{E}' \to \mathrm{L}$ the double cover branched on the four points of $\mathrm{L} \cap (\Gamma - \mathrm{L})$, then $\mathcal{F}$ is the double cover of $\mathcal{E}'$. 
\end{lemma}
We have the following:
\begin{proposition}
\label{cor:genus_5_bielliptic}
Let  $\mathcal{F}_{x_0}$ be a smooth curve of genus five and $\mathrm{L}$, $\mathrm{C} \subset \Gamma^+$ as above. The double cover of $\mathbb{P}^1$ branched on the four point of $\mathrm{L} \cap (\mathrm{C} \cup \Gamma^-)$ is an irreducible component of $W^1_4(\mathcal{F}_{x_0})$, and $\mathcal{E}'_{x_0}=\phi^{-1}(\mathrm{L})$ is the elliptic curve
\beq
\label{eqn:JacFib_dual2}
 \mathcal{E}'_{x_0}: \quad  y^2 = S(\gamma) \, P(x_0) \, x^3 +\Big( 2 \, \mu \, P(x_0) + (\gamma-\delta)^2 \, Q(x_0)  \Big) \, x^2 +  S(\delta) \, P(x_0) \, x \,,
\eeq
where $S(\xi)$, $P(x)$, $Q(x)$, $\mu$ are defined in Section~\ref{ssec:AJM} and Equation~(\ref{eqn:moduli}). Moreover, each curve $\mathcal{F}_{x_0}$ admits a double cover $\pi'_{x_0}: \mathcal{F}_{x_0} \to \mathcal{E}'_{x_0}$.
\end{proposition}
\begin{proof}
One checks that $\Gamma^+ = \mathrm{L} \cup  \mathrm{C}$ decomposes into the line component $\mathrm{L}$ and another irreducible conic $\mathrm{C}$ with $\mathrm{L} = \mathrm{V}(\alpha_0)$ and $\mathrm{C}=\mathrm{V}(q(\alpha_0,\alpha_1,\alpha_2))$ for a certain conic $q$. The double cover of $\mathbb{P}^1$ branched on the four points of $\Gamma^- \cap \mathrm{C}$ is then given by
\beq
\begin{split}
&  \left. \Big(\big(\alpha_0^2-4\alpha_1\alpha_2\Big) \; q(\alpha_0,\alpha_1,\alpha_2)\Big)\right|_{[\alpha_0:\alpha_1:\alpha_2]=[0:x:z]} \\
= & \; S(\gamma) \, P(x_0) \, x^3 +\Big( 2 \, \mu \, P(x_0) + (\gamma-\delta)^2 \, Q(x_0)  \Big) \, x^2z +  S(\delta) \, P(x_0) \, x z^2 \,,
\end{split} 
\eeq
which agrees with Equation~(\ref{eqn:JacFib_dual2}) in the affine coordinate chart $z=1$.
\end{proof}
We also introduce $\mathcal{Z}'= \coprod_{x_0} \mathcal{E}'_{x_0}$, i.e., the total space of the elliptic fibration obtained by varying the parameter $x_0$ in Equation~(\ref{eqn:JacFib_dual2}). We have the following:
\begin{corollary}
The total space of the elliptic fibration $\mathcal{Z}'= \coprod_{x_0} \mathcal{E}'_{x_0}$ is birationally equivalent to the Jacobian elliptic fibration in Equation~(\ref{eqn:JacFib_dual}).
\end{corollary}
\begin{proof}
Using $\mu^2-\nu \kappa = S(\gamma) \, S(\delta)$ the map $(x,y) \mapsto (4 P(x_0) \, S(\gamma) \, x , \ 2\sqrt{2} \, P(x_0) \, S(\gamma) \, y)$ provides a birationally equivalence of the Jacobian elliptic fibration in Equation~(\ref{eqn:JacFib_dual}) with Equation~(\ref{eqn:JacFib_dual2}).
\end{proof}
\par We also describe the remaining irreducible component of $W^1_4(\mathcal{F}_{x_0})$ which is $\widetilde{\mathcal{C}}_{x_0}=\phi^{-1}(\mathrm{C})$. There is a classification of the singular fibers of pencils of curves of genus two due to Namikawa and Ueno \cite{MR0369362} analogous to the Kodaira classification of singular fibers of Jacobian elliptic fibrations \cite{MR0184257}. We have the following:
\begin{proposition}
Let  $\mathcal{F}_{x_0}$ be a smooth curve of genus five and $\mathrm{L}$, $\mathrm{C} \subset \Gamma^+$ as above. The double cover of $\mathbb{P}^1\cong \mathrm{C}$ branched on the six points of $\mathrm{C} \cap (\mathrm{L} \cup \Gamma^-)$ is an irreducible component of $W^1_4(\mathcal{F}_{x_0})$, and $\widetilde{\mathcal{C}}_{x_0}=\phi^{-1}(\mathrm{C})$ is the curve of genus two
\beq
\label{eqn:pencilG2}
\begin{split}
\widetilde{\mathcal{C}}_{x_0}: &\quad  \eta^2 =  \Big( \big( \gamma \xi + 1 \big) P(x_0) + \xi \,Q(x_0) \Big) \Big( S(\gamma) \, \xi^3 + \big(3 \gamma^2 + S'(0) \big) \xi^2 + 3 \gamma  \xi \Big)\\
 \times & \Big( \big( S(\gamma) \, \xi^2 + (2 \gamma^2 + \gamma\delta + S'(0)) \, \xi^2 + \gamma + \delta \big) \, P(x_0) - \big( (\gamma-\delta) \xi +1 \big) \, Q(x_0) \Big) \,,
\end{split}
\eeq
where $S(\xi)$, $P(x)$, $Q(x)$ are defined in Section~\ref{ssec:AJM} and Equation~(\ref{eqn:moduli}). Moreover, the family $\widetilde{\mathcal{C}}_{x_0}$ has six singular fibers of Namikawa-Ueno type $[I_0-I_0^*-0]$ over the roots of $[ P, Q ]_{x_0}$.  
\end{proposition}
\begin{proof}
One checks that the component $\mathrm{C}=\mathrm{V}(q(\alpha_0,\alpha_1,\alpha_2)) \subset \Gamma^+$ contains the rational point $[\alpha_0: \alpha_1: \alpha_2] = [ -2: 1 :1]$, and a rational parametrization is given by setting $\alpha_1 = \alpha_2 +(-1 + (\gamma-\delta) \xi/2) (\alpha_0 + 2 \alpha_2)$. We obtain the double cover of $\mathbb{P}^1 \cong \mathrm{C}$ branched on the six intersection points of $\mathrm{L} \cup \Gamma^-$, by substituting the rational parametrization into $\alpha_0 (\alpha_0^2 - 4 \alpha_1 \alpha_2)$. We obtain
\beq
\label{eqn:pencilG2_proof}
\begin{split}
 \eta^2 = &  \underbrace{\Big( \big( \gamma \xi + 1 \big) P(x_0) + \xi \,(x_0) \Big)}_{=: \, p_1(\xi)} \underbrace{\Big( S(\gamma) \, \xi^3 + \big(3 \gamma^2 + S'(0) \big) \xi^2 + 3 \gamma  \xi \Big)} _{=: \, p_2(\xi)}\\
 \times & \underbrace{\Big( \big( S(\gamma) \, \xi^2 + (2 \gamma^2 + \gamma\delta + S'(0)) \, \xi^2 + \gamma + \delta \big) \, P(x_0) - \big( (\gamma-\delta) \xi +1 \big) \, Q(x_0) \Big)}_{=: \, p_3(\xi)} \,.
\end{split}
\eeq
We compute the following resultants
\beqn
\begin{split}
 \operatorname{Res}_\xi\Big( p_1(\xi) , p_2(\xi) \Big) & = S(\gamma) \, S(\delta) \, \Big( S(0) \, P(x_0)^3 - S'(0) \, P(x_0)^2 Q(x_0)- Q(x_0)^3 \Big) \,,\\
 \operatorname{Res}_\xi\Big( p_1(\xi) , p_3(\xi) \Big) & = \operatorname{Res}_\xi\Big( p_2(\xi) , p_3(\xi) \Big)  = S(0) \, P(x_0)^3 - S'(0) \, P(x_0)^2 Q(x_0) - Q(x_0)^3 \,.
\end{split} 
\eeqn 
We then compute the Igusa-Clebsch invariants of the curve of genus two, using the same normalization as in \cites{MR3712162,MR3731039}.  Denoting the Igusa-Clebsch invariants of the curve of genus two in Equation~(\ref{eqn:pencilG2_proof})  by $[ I_2 : I_4 : I_6 : I_{10} ] \in \mathbb{P}(2,4,6,10)$, one checks that for
\beq 
 \epsilon= 4 \Big(S(0) \, P(x_0)^3 - S'(0) \, P(x_0)^2 Q(x_0)- Q(x_0)^3\Big) = \Big(\,  [ P, Q ]_{x_0} \Big)^2 \,.
\eeq
we obtain
\beq
[ I_2 : I_4 : I_6 : I_{10} ]  = \Big[ I_2 \ : \ \epsilon^2  I'_4\ : \ \epsilon^2 I'_6 \ : \ \epsilon^6 I'_{10}\Big] \,,
\eeq 
where $I_2, I'_4,  I'_6, I'_{10}$ are polynomials in $x_0$ that do not have a common factor with $[ P, Q ]_{x_0}$. Using the results of Namikawa and Ueno \cite{MR0369362} we conclude that a local model for $\widetilde{\mathcal{C}}_{x_0}$ near $\epsilon=0$ is given by
\beq
 \eta^2 = \Big( \xi^3 + \alpha \, \xi + 1\Big)  \Big( \xi^3 + \epsilon^2 \beta \, \xi + \epsilon^3 \Big)  \,,
\eeq
where $\alpha, \beta$ are suitable rational functions that do not vanish for $\epsilon=0$.
\end{proof}

\section{Proof of the main results}
\label{section4}
We have the following:
\begin{proposition}
\label{prop:main}
For
\beq
\label{eqn:polynomials}
\begin{split}
  P(x)  =x^4 -\Lambda_1 x^2  + 1 \,, \qquad
  R(x, x_0)   = x^2 x_0^2 - \frac{\Lambda_1}{6} \big( x^2 + 4 x_0 x + x_0^2 \big) + 1 \,, \\
  R_1(x, x_0) = - \frac{2}{3} \Lambda_1 \, x^2 x_0^2 + \left( 2 - \frac{5}{18} \Lambda_1^2\right) \, x x_0 +  \left( 1 - \frac{1}{36} \Lambda_1^2\right) \big(x^2+x_0^2\big) - \frac{2}{3} \Lambda_1 \,,
\end{split}  
\eeq
and parameters 
\beq
\label{eqn:moduli_gd_intro2}
 \gamma+\delta= - \frac{c_1}{3\kappa_{1,5}\kappa_{2,3}c_2} \,, \qquad \gamma\delta = \frac{c_0}{9\lambda_1\lambda_2\lambda_3c_2} \,,
\eeq 
with $c_0, c_1, c_2$ given by either Equation~(\ref{eqn:moduli_gd_1}) or Equation~(\ref{eqn:moduli_gd_2}), the Jacobian elliptic K3 surfaces $\mathcal{Y}$ and $\mathcal{Z}$ in Equation~(\ref{eqn:B12}) and Equation~(\ref{eqn:JacFib}) coincide for $v=x_0$. In particular, the Jacobian elliptic fibrations realize the fibration from Proposition~\ref{prop:Garbagnati} where $\mathrm{K}_0, \dots, \mathrm{K}_3$ are the divisor classes from Corollary~\ref{cor:sections_and_divisors}. The same applies to the Jacobian elliptic K3 surfaces $\mathcal{X}'$ and $\mathcal{Z}'$ in Equation~(\ref{XXsurface}) and Equation~(\ref{eqn:JacFib_dual})/Equation~(\ref{eqn:JacFib_dual2}), respectively.
\end{proposition}
\begin{proof}
We have the two (pairs of) points $\pm q_\gamma,  \pm q_\delta \in \mathcal{E}$ on the elliptic curve in Equation~(\ref{eqn:EC}) with coordinates $(\xi_{q_\gamma} = \gamma, \, \pm \eta_{q_\gamma})$ with $\eta_{q_\gamma}^2 = S(\gamma)$ and $(\xi_{q_\delta} = \delta, \, \pm \eta_{q_\delta})$ with $\eta_{q_\delta}^2 = S(\delta)$. One then checks that $\mu^2-\kappa \nu =S(\gamma) S(\delta)$. By a rescaling one obtains from Equation~(\ref{eqn:JacFib}) the Weierstrass model
\beq
\label{eqn:JacFib2}
 Y^2 = X^3 -2 \left(\frac{\mu \,P(x_0)}{\eta_{q_\gamma}\eta_{q_\delta}} + \frac{\nu \, Q(x_0)}{\eta_{q_\gamma}\eta_{q_\delta}}  \right) \, X^2 +  \left( \left(\frac{\mu \,P(x_0)}{\eta_{q_\gamma}\eta_{q_\delta}}  + \frac{\nu\,Q(x_0)}{\eta_{q_\gamma}\eta_{q_\delta}}   \right)^2 - P(x_0)^2 \right) X .
\eeq
with
\beq
\begin{split}
  \frac{\mu}{\eta_{q_\gamma}\eta_{q_\delta}}= \frac{\xi_{q_\gamma+q_\delta} + \xi_{q_\gamma-q_\delta}}{\xi_{q_\gamma+q_\delta} - \xi_{q_\gamma-q_\delta}}  \,, \qquad  \frac{\nu}{\eta_{q_\gamma}\eta_{q_\delta}} = \frac{2}{\xi_{q_\gamma+q_\delta} - \xi_{q_\gamma-q_\delta}} \,.
\end{split}
\eeq
The choice of sign $\pm \eta_{q_\gamma}$ and $\pm \eta_{q_\gamma}$ does not matter as it can always be absorbed in a rescaling $(X,Y) \mapsto (-X,iY)$. Plugging in $P(x)$ and $\gamma, \delta$, one checks that the Weierstrass models in Equation~(\ref{eqn:B12}) and Equation~(\ref{eqn:JacFib2}) are identical for $v=x_0$. In particular, it follows that Equation~(\ref{eqn:moduli_gd_1}) and~(\ref{eqn:moduli_gd_2}) are the only solutions that make the Jacobian elliptic K3 surfaces $\mathcal{Y}$ and $\mathcal{Z}$ coincide up to a sign change $\pm \kappa_p$.
\end{proof}
\par According to Remark~\ref{rem:isog_curves} there are exactly three inequivalent G\"opel groups containing a given point of order two $p_{46} \in \operatorname{Jac}(\mathcal{C})[2]$. The point of order two determines a rational double cover $\phi_{\Delta_{p_{46}}}: \mathcal{Y} \dasharrow \mathcal{X}=\operatorname{Kum}(\operatorname{Jac} \mathcal{C})$. Remark~\ref{rem:choice_torsion_sections} shows that the full G\"opel group $G' \ni p_{46}$ also determines Weierstrass models in Equation~(\ref{kummer_middle_ell_p_W}) and Equation~(\ref{eqn:B12}), together with a marked 2-torsion section. Following Remark~\ref{rem:choice_even_eight}, the marked 2-torsion section on $\mathcal{Y}$ defines an even eight $\Delta_{G'}$ of exceptional curves on $\mathcal{Y} = \operatorname{Kum}(\mathsf{B})$ which in turn determines a rational double cover $\phi_{\Delta_{G'}}: \mathcal{X}' \dasharrow \mathcal{Y}$. Thus, the Kummer surface $\mathcal{X}' =\operatorname{Kum}(\operatorname{Jac} \mathcal{C}')$ is obtained from the G\"opel group $G'$ such that $\operatorname{Jac}(\mathcal{C}')=\operatorname{Jac}(\mathcal{C})/G'$. As $\phi_{\Delta_{G'}}$ is associated with a van~Geemen-Sarti involution this establishes a Jacobian elliptic fibration on $\mathcal{X}'$.
\subsection{Proof of Theorem~\ref{main1}}
\label{proof_thm}
Let us first explain the rescaling that yields the pencil $\widetilde{\mathcal{D}}_{x_0}$ in Equation~(\ref{eqn:genus-three_intro}) from Equation~(\ref{eqn:genus-three_blow-up}) using the parameters in Equation~(\ref{eqn:polynomials}) and Equation~(\ref{eqn:moduli_gd_intro2}). We set $x= \sqrt{l} X$, $x_0= \sqrt{l} t$, $\tilde{z} = Z/\sqrt{9 c_2 l}$ with $l=\kappa_{1,5} \kappa_{2,3}$, and 
\beq
\begin{array}{rclcrcl}
 p(X,Y) & = & P\left(\frac{\sqrt{l} X}{Y} \right) Y^4 \,, &&
 \Delta^{(t)}(X,Y) & = &  \big(X - t \, Y\big)^2 \,,\\
 r^{(t)}(X, Y) & = &6 \, R\left( \frac{\sqrt{l} X}{Y}, \sqrt{l} t\right) Y^2 \,, &&
 r^{(t)}_1(X, Y) & = & -36 \, l \,  R_1\left( \frac{\sqrt{l}X}{Y}, \sqrt{l} t \right) Y^2 \,,\\
\end{array} 
\eeq
multiply Equation~(\ref{eqn:genus-three_blow-up}) with $(9 c_2 l)^2$ to obtain the equation for $\mathcal{D}_t :=\widetilde{\mathcal{D}}_{x_0}$ given by
\beq
\label{eqn:genus-three_end}
 \mathcal{D}_t: \quad p_0^{(t)} \, Z^4 +  \Big(c_2 r_1^{(t)} +  c_1  r^{(t)} + c_0 \Delta^{(t)} \Big)\, Z^2 + 9 \, \Big(c_1^2- 4c_0c_2\Big)  \, p =0 \,,
\eeq
where $c_0, c_1, c_2$ are given by Equation~(\ref{eqn:moduli_gd_1}) or Equation~(\ref{eqn:moduli_gd_2}), and $\kappa_p^2=\lambda_2\lambda_3$ or $\kappa_p^2=\lambda_1$, respectively. Note that changing from $\mathcal{D}_{x_0}$ to $\widetilde{\mathcal{D}}_{x_0}$ in Equation~(\ref{eqn:genus-three_blow-up}) does not affect the smooth fibers -- this also applies to Sections~\ref{proof2}/\ref{proof3}.  In the following, we will restrict ourselves to the case of Equation~(\ref{eqn:moduli_gd_2}), i.e., $\kappa_{1,5}^2=\lambda_1$, and $\mathfrak{M}'=\mathfrak{M}'_{p_{15}}$. The other case is completely analogous. 
\par (1) It follows from Proposition~\ref{prop:main} that the Jacobian elliptic K3 surfaces $\mathcal{Z}$ and $\mathcal{Y}$ in Equation~(\ref{eqn:JacFib}) and Equation~(\ref{eqn:B12}) coincide for $x_0=v$. It was proven in Proposition~\ref{lem:B12} that the K3 surface $\mathcal{Y}$ is the Kummer surface of an abelian surface $\mathsf{B}_{p_{46}}$ with a polarization of type $(1,2)$. Proposition~\ref{prop:pencil} then shows that the pencil of curves of genus three $\mathcal{D}_{x_0}$ is obtained as double cover of  $\mathcal{Z}= \coprod_{x_0} \mathcal{Q}_{x_0}$ branched on the divisor classes $\mathrm{K}_0, \dots, \mathrm{K}_3$ in Corollary~\ref{cor:sections_and_divisors}.  According to Theorem~\ref{thm:Barth}, this is precisely the pencil on $\mathsf{B}$ realizing the linear system $|\mathscr{V}|$ for the $(1,2)$-polarization on $\mathsf{B}$ given by an ample symmetric line bundle $\mathscr{V}$ with $\mathscr{V}^2 =4$. Thus, the claim follows.
\par~(2) and (3) One checks that the discriminant in Equation~(\ref{eqn:discriminant}) vanishes for $t^2 = \lambda_1, \lambda_2\lambda_3$, $t^2 = \lambda_2, \lambda_1\lambda_3$, and $t^2 = \lambda_3, \lambda_1\lambda_2$. The proof then follows from Proposition~\ref{prop:normalization} together with Proposition~\ref{prop:isog_genus2_curve}.
\par~(4) The claims follows from Proposition~\ref{prop:hyperelliptic} as follows: one checks that the roots of $[ P, Q ]_{x_0}=0$ are given by $t^2 = 0, \pm \lambda_1 \lambda_2 \lambda_3, \infty$ and that for $\gamma, \delta$ given by Equations~(\ref{eqn:moduli_gd_intro2}), the condition $2(q_\gamma \pm q_\delta) \not = 0$ is satisfied. \qed
\subsection{Proof of Theorem~\ref{main2a}}
\label{proof2}
The point $p_{46} \in \operatorname{Jac}(\mathcal{C})[2]$ determines a 2-isogeny $\Phi: \mathsf{B} \to \operatorname{Jac}(\mathcal{C})$ which covers  $\phi_{\Delta_{p_{46}}}: \mathcal{Y} \dasharrow \mathcal{X}$. The Weierstrass model (with marked 2-torsion) on $\mathcal{Y}\cong \mathcal{Z}$ in Equation~(\ref{eqn:B12}) is then used in Proposition~\ref{prop:pencil} to construct the pencil $\mathcal{D}_{x_0}$ of bielliptic curves of genus three realizing $|\mathscr{V}|$ where $\mathscr{V}$ is the polarization line bundle on $\mathsf{B}$ induced by pull-back. The equivalent pencil $\widetilde{\mathcal{D}}_{x_0}$ has the property that the normalization of four singular fibers is given by the $(2,2)$-isogenous curve $\mathcal{C}'$; see Proposition~\ref{prop:normalization}. The normal form for $\widetilde{\mathcal{D}}_{x_0}$ in Equation~(\ref{eqn:g3curve_AG}) also determines an unramified double cover $\rho'_{x_0}: \mathcal{F}_{x_0} \to \widetilde{\mathcal{D}}_{x_0}$ by a non-hyperelliptic curve of genus five $ \mathcal{F}_{x_0} $; see Remark~\ref{rem:hrouping_bielliptic_tangents}. Its Prym variety is the principally polarized abelian surface  $\operatorname{Prym}(\mathcal{F}_{x_0}, \rho'_{x_0}) = \operatorname{Jac}(\mathcal{C}')$; see Corollary~\ref{cor:genus_5_prym}.  Proposition~\ref{cor:genus_5_bielliptic} proves that the curves of genus five $\mathcal{F}_{x_0}$ also admit a double cover $\pi'_{x_0}: \mathcal{F}_{x_0} \to \mathcal{E}'_{x_0}$ onto the elliptic curves $\mathcal{E}'_{x_0}$ such that $\mathcal{Z}'= \coprod_{x_0} \mathcal{E}'_{x_0}$ is the Jacobian elliptic fibration~(\ref{XXsurface}) on the Kummer surface $\mathcal{X}'=\operatorname{Kum}(\operatorname{Jac} \mathcal{C}')$. We have the following:
\begin{lemma}
\label{lem:polarizing_divisor}
Assuming Equations~(\ref{eqn:polynomials}) and~(\ref{eqn:moduli_gd_intro2}), the curves of genus five $\mathcal{F}_{x_0}$ admitting the unramified cover $\rho'_{x_0}: \mathcal{F}_{x_0} \to \widetilde{\mathcal{D}}_{x_0}$ form a pencil on $\operatorname{Jac}(\mathcal{C}')$, and $\mathcal{F}_{x_0}$ embeds into  $\operatorname{Prym}{(\mathcal{F}_{x_0}, \rho'_{x_0})} \cong \operatorname{Jac}(\mathcal{C}')$ as a curve of self-intersection eight.
\end{lemma}
\begin{proof}
The proof follows from Lemma~\ref{lem:APembedding}. Since $\mathcal{F}_{x_0}$ represents the pull-back of a theta divisor via a degree-two map the self-intersection is eight.
\end{proof}
We make the following:
\begin{remark}
Geometrically, $\widetilde{\mathcal{D}}_{x_0}$ is obtained as follows: given the curve of genus two $\mathcal{C}'$ and its Kummer quartic $\mathrm{K}'=\operatorname{Jac}(\mathcal{C}')/\langle -{\rm id} \rangle \subset \mathbb{P}^3$, we can always find a plane $\mathrm{V}_{x_0} \subset \mathbb{P}^3$ such that $\widetilde{\mathcal{D}}_{x_0} = \mathrm{V}_{x_0} \cap \mathrm{K}'$ is a non-singular quartic curve not meeting the ramification locus of $\pi: \operatorname{Jac}(\mathcal{C}') \to \mathrm{K}'$.  Then, $\mathcal{F}_{x_0}=\pi^{-1}(\mathcal{C}')$ is an unramified double cover of $\mathcal{C}'$ and connected, whence of genus five.  This model of $\widetilde{\mathcal{D}}_{x_0}$ as a plane section of $\mathrm{K}'$ also determines the 28 bitangents of $\widetilde{\mathcal{D}}_{x_0}$. The tropes on $\mathrm{K}'$ cut out sixteen bitangents; the remaining twelve come in pairs from singular conics in $\Gamma$. Our model for $\widetilde{\mathcal{D}}_{x_0}$ in Equation~(\ref{eqn:g3curve_AG}) has only six rational tangents over $\mathfrak{M}'$; see Remark~\ref{rem:dual_Goepel} and \cite{Clingher:2018aa}*{Table~3}. However, there are additional bitangents coming from singular conics which determine $q_1, q_2$ in Equation~(\ref{eqn:normal_form}): they are in general not rational over $\mathfrak{M}'$, but their product always is. In fact, only $\gamma + \delta$ and $\gamma\delta$ in Equation~(\ref{eqn:g5_quadrics}) are rational over $\mathfrak{M}'$; see Equation~(\ref{eqn:moduli_gd_intro2}).
\end{remark}
\par We use the same identification of moduli as in Section~\ref{proof_thm}. In addition, we rescale $V, W \mapsto V/(3\sqrt{l c_2}), W/(3\sqrt{l c_2})$ with $l=\kappa_{1,5} \kappa_{2,3}$. Note that $l=0$ or $c_2=0$ implies that $\mathcal{C}$ is singular. We then introduce the parameters $e=-\gamma/(3l)$ and $f=-\delta/(3l)$ such that Equation~(\ref{eqn:moduli_gd_intro2}) becomes $e+f=c_1/c_2$, $e f=c_0/c_2$ and Equation~(\ref{eqn:genus_5_curve}) becomes Equation~(\ref{eqn:genus_5_curve_intro}).  Interchanging $e$ and $f$ amounts to the changing the sign of $\sqrt{c_1^2-4c_0c_2}$ which is easily checked to correspond to a sign change $\pm \kappa_{1,5}$ or, equivalently, swapping the two sheets of the double cover $\mathfrak{M}' \to \mathfrak{M}$. A computation then shows that the curves of genus two in Proposition~\ref{prop:isog_genus2_curve} and Proposition~\ref{prop:genus_5_prym}/Corollary~\ref{cor:genus_5_prym} coincide. Upon re-scaling of variables we obtain Equation~(\ref{eqn:genus-two-dual_intro}).  The fact that the curve is an Abel-Prym embedding and also  bielliptic was proved in Lemma~\ref{lem:APembedding} and Proposition~\ref{cor:genus_5_bielliptic}; finally, we use Lemma~\ref{lem:polarizing_divisor}. \qed
\subsection{Proof of Corollary~\ref{main2}}
\label{proof3}
Theorem~\ref{main1} already proves that for a smooth curve $\mathcal{D}_{x_0}$ the Prym variety $\operatorname{Prym}(\mathcal{D}_{x_0},\pi_{x_0})$ with its polarization of type $(1,2)$ is 2-isogenous to the principally polarized Jacobian variety $\operatorname{Jac}(\mathcal{C})$. The proof of the corollary then follows from Lemma~\ref{lem:product_decomp} and Proposition~\ref{prop:decomp} after observing that for $x_0=0$ the curve $\mathcal{D}_{x_0}$ is smooth and for its bielliptic quotient $\mathcal{Q}_{x_0} = \mathcal{D}_{x_0} /\langle \jmath \rangle$ the Jacobian $\operatorname{Jac}(\mathcal{Q}_{x_0})$ has the same $j$-invariant as the one in Equation~(\ref{eqn:j-inv}). The same argument applies for $x_0=\infty$. \qed
\bibliographystyle{amsplain}
\bibliography{ref}{}

\begin{bibdiv}
\begin{biblist}

\bib{MR912838}{article}{
      author={Adler, Mark},
      author={van Moerbeke, Pierre},
       title={The intersection of four quadrics in {${\bf P}^6$}, abelian
  surfaces and their moduli},
        date={1987},
        ISSN={0025-5831},
     journal={Math. Ann.},
      volume={279},
      number={1},
       pages={25\ndash 85},
         url={https://mathscinet.ams.org/mathscinet-getitem?mr=912838},
      review={\MR{912838}},
}

\bib{MR923636}{article}{
      author={Adler, Mark},
      author={van Moerbeke, Pierre},
       title={The {K}owalewski and {H}\'{e}non-{H}eiles motions as {M}anakov
  geodesic flows on {${\rm SO}(4)$}---a two-dimensional family of {L}ax pairs},
        date={1988},
        ISSN={0010-3616},
     journal={Comm. Math. Phys.},
      volume={113},
      number={4},
       pages={659\ndash 700},
         url={https://mathscinet.ams.org/mathscinet-getitem?mr=923636},
      review={\MR{923636}},
}

\bib{MR770932}{book}{
      author={Arbarello, E.},
      author={Cornalba, M.},
      author={Griffiths, P.~A.},
      author={Harris, J.},
       title={Geometry of algebraic curves. {V}ol. {I}},
      series={Grundlehren der Mathematischen Wissenschaften [Fundamental
  Principles of Mathematical Sciences]},
   publisher={Springer-Verlag, New York},
        date={1985},
      volume={267},
        ISBN={0-387-90997-4},
         url={https://mathscinet.ams.org/mathscinet-getitem?mr=770932},
      review={\MR{770932}},
}

\bib{MR1816214}{article}{
      author={Bardelli, F.},
      author={Del~Centina, A.},
       title={Bielliptic curves of genus three: canonical models and moduli
  space},
        date={1999},
        ISSN={0019-3577},
     journal={Indag. Math. (N.S.)},
      volume={10},
      number={2},
       pages={183\ndash 190},
         url={http://dx.doi.org/10.1016/S0019-3577(99)80015-9},
      review={\MR{1816214}},
}

\bib{MR946234}{incollection}{
      author={Barth, Wolf},
       title={Abelian surfaces with {$(1,2)$}-polarization},
        date={1987},
   booktitle={Algebraic geometry, {S}endai, 1985},
      series={Adv. Stud. Pure Math.},
      volume={10},
   publisher={North-Holland, Amsterdam},
       pages={41\ndash 84},
      review={\MR{946234}},
}

\bib{MR2729013}{article}{
      author={Bastianelli, F.},
      author={Pirola, G.~P.},
      author={Stoppino, L.},
       title={Galois closure and {L}agrangian varieties},
        date={2010},
        ISSN={0001-8708},
     journal={Adv. Math.},
      volume={225},
      number={6},
       pages={3463\ndash 3501},
         url={https://doi.org/10.1016/j.aim.2010.06.006},
      review={\MR{2729013}},
}

\bib{MR572974}{article}{
      author={Beauville, Arnaud},
       title={Prym varieties and the {S}chottky problem},
        date={1977},
        ISSN={0020-9910},
     journal={Invent. Math.},
      volume={41},
      number={2},
       pages={149\ndash 196},
         url={https://mathscinet.ams.org/mathscinet-getitem?mr=572974},
      review={\MR{572974}},
}

\bib{beshaj2014decomposition}{inproceedings}{
      author={Beshaj, Lubjana},
      author={Shaska, Tony},
       title={Decomposition of some {J}acobian varieties of dimension 3},
organization={Springer},
        date={2014},
   booktitle={International conference on artificial intelligence and symbolic
  computation},
       pages={193\ndash 204},
}

\bib{MR2062673}{book}{
      author={Birkenhake, Christina},
      author={Lange, Herbert},
       title={Complex abelian varieties},
     edition={Second Edition},
      series={Grundlehren der Mathematischen Wissenschaften [Fundamental
  Principles of Mathematical Sciences]},
   publisher={Springer-Verlag, Berlin},
        date={2004},
      volume={302},
        ISBN={3-540-20488-1},
         url={http://dx.doi.org/10.1007/978-3-662-06307-1},
      review={\MR{2062673}},
}

\bib{MR970659}{article}{
      author={Bost, Jean-Beno\^it},
      author={Mestre, Jean-Fran\c{c}ois},
       title={Moyenne arithm{\'e}tico-g{\'e}om{\'e}trique et p{\'e}riodes des
  courbes de genre {$1$} et {$2$}},
        date={1988},
        ISSN={0224-8999},
     journal={Gaz. Math.},
      number={38},
       pages={36\ndash 64},
      review={\MR{970659}},
}

\bib{MR2406115}{article}{
      author={Bruin, Nils},
       title={The arithmetic of {P}rym varieties in genus 3},
        date={2008},
        ISSN={0010-437X},
     journal={Compos. Math.},
      volume={144},
      number={2},
       pages={317\ndash 338},
         url={https://mathscinet.ams.org/mathscinet-getitem?mr=2406115},
      review={\MR{2406115}},
}

\bib{MR4015343}{article}{
      author={Clingher, A.},
      author={Malmendier, A.},
      author={Shaska, T.},
       title={Six line configurations and string dualities},
        date={2019},
        ISSN={0010-3616},
     journal={Comm. Math. Phys.},
      volume={371},
      number={1},
       pages={159\ndash 196},
         url={https://mathscinet.ams.org/mathscinet-getitem?mr=4015343},
      review={\MR{4015343}},
}

\bib{MR2824841}{article}{
      author={Clingher, Adrian},
      author={Doran, Charles~F.},
       title={Note on a geometric isogeny of {K}3 surfaces},
        date={2011},
        ISSN={1073-7928},
     journal={Int. Math. Res. Not. IMRN},
      number={16},
       pages={3657\ndash 3687},
      review={\MR{2824841}},
}

\bib{Clingher:2017aa}{article}{
      author={Clingher, Adrian},
      author={Malmendier, Andreas},
       title={On the geometry of (1,2)-polarized {K}ummer surfaces},
        date={2017},
     journal={arXiv:1704.04884 [math.AG]},
         url={https://arxiv.org/pdf/1704.04884},
}

\bib{MR3995925}{article}{
      author={Clingher, Adrian},
      author={Malmendier, Andreas},
       title={Nikulin involutions and the {CHL} string},
        date={2019},
        ISSN={0010-3616},
     journal={Comm. Math. Phys.},
      volume={370},
      number={3},
       pages={959\ndash 994},
         url={https://mathscinet.ams.org/mathscinet-getitem?mr=3995925},
      review={\MR{3995925}},
}

\bib{Clingher:2018aa}{article}{
      author={Clingher, Adrian},
      author={Malmendier, Andreas},
       title={Normal forms for {K}ummer surfaces},
        date={2019},
     journal={London Mathematical Society Lecture Note Series},
      volume={2},
      number={459},
       pages={107\ndash 162},
         url={https://arxiv.org/pdf/1809.02003},
}

\bib{CMS:2019}{article}{
      author={Clingher, Adrian},
      author={Malmendier, Andreas},
      author={Shaska, Tony},
       title={On isogenies among certain abelian varieties},
        date={2019},
     journal={arXiv:1901.09846 [math.AG]},
}

\bib{MR932781}{article}{
      author={Cornalba, Maurizio},
       title={On the locus of curves with automorphisms},
        date={1987},
        ISSN={0003-4622},
     journal={Ann. Mat. Pura Appl. (4)},
      volume={149},
       pages={135\ndash 151},
         url={http://dx.doi.org/10.1007/BF01773930},
      review={\MR{932781}},
}

\bib{MR3389883}{article}{
      author={Couveignes, Jean-Marc},
      author={Ezome, Tony},
       title={Computing functions on {J}acobians and their quotients},
        date={2015},
     journal={LMS J. Comput. Math.},
      volume={18},
      number={1},
       pages={555\ndash 577},
         url={https://mathscinet.ams.org/mathscinet-getitem?mr=3389883},
      review={\MR{3389883}},
}

\bib{MR2457735}{incollection}{
      author={Dolgachev, I.},
      author={Lehavi, D.},
       title={On isogenous principally polarized abelian surfaces},
        date={2008},
   booktitle={Curves and abelian varieties},
      series={Contemp. Math.},
      volume={465},
   publisher={Amer. Math. Soc., Providence, RI},
       pages={51\ndash 69},
         url={https://doi.org/10.1090/conm/465/09100},
      review={\MR{2457735}},
}

\bib{dolgachev2014endomorphisms}{article}{
      author={Dolgachev, Igor},
       title={Endomorphisms of complex abelian varieties,},
     journal={{L}ecture {N}otes, {M}ilan 2014},
  pages={http://www.math.las.umich.edu$/\!\sim$idolga$/$lecturenotes.html},
}

\bib{MR1188194}{incollection}{
      author={Donagi, Ron},
       title={The fibers of the {P}rym map},
        date={1992},
   booktitle={Curves, {J}acobians, and abelian varieties ({A}mherst, {MA},
  1990)},
      series={Contemp. Math.},
      volume={136},
   publisher={Amer. Math. Soc., Providence, RI},
       pages={55\ndash 125},
         url={https://mathscinet.ams.org/mathscinet-getitem?mr=1188194},
      review={\MR{1188194}},
}

\bib{MR1736231}{article}{
      author={Donagi, Ron},
      author={Livn\'{e}, Ron},
       title={The arithmetic-geometric mean and isogenies for curves of higher
  genus},
        date={1999},
        ISSN={0391-173X},
     journal={Ann. Scuola Norm. Sup. Pisa Cl. Sci. (4)},
      volume={28},
      number={2},
       pages={323\ndash 339},
         url={https://mathscinet.ams.org/mathscinet-getitem?mr=1736231},
      review={\MR{1736231}},
}

\bib{MR3798190}{article}{
      author={Enolski, V.~Z.},
      author={Fedorov, Yu.~N.},
       title={Algebraic description of {J}acobians isogeneous to certain {P}rym
  varieties with polarization {$(1,2)$}},
        date={2018},
        ISSN={1058-6458},
     journal={Exp. Math.},
      volume={27},
      number={2},
       pages={147\ndash 178},
         url={https://mathscinet.ams.org/mathscinet-getitem?mr=3798190},
      review={\MR{3798190}},
}

\bib{Garbagnati08}{thesis}{
      author={Garbagnati, Alice},
       title={Symplectic automorphisms on {K}3 surfaces},
        type={Ph.D. Thesis},
        date={2008},
}

\bib{MR3010125}{article}{
      author={Garbagnati, Alice},
      author={Sarti, Alessandra},
       title={On symplectic and non-symplectic automorphisms of {K}3 surfaces},
        date={2013},
        ISSN={0213-2230},
     journal={Rev. Mat. Iberoam.},
      volume={29},
      number={1},
       pages={135\ndash 162},
         url={https://doi.org/10.4171/RMI/716},
      review={\MR{3010125}},
}

\bib{MR2166182}{article}{
      author={van Geemen, Bert},
       title={Some remarks on {B}rauer groups of {$K3$} surfaces},
        date={2005},
        ISSN={0001-8708},
     journal={Adv. Math.},
      volume={197},
      number={1},
       pages={222\ndash 247},
         url={https://mathscinet.ams.org/mathscinet-getitem?mr=2166182},
      review={\MR{2166182}},
}

\bib{MR2274533}{article}{
      author={van Geemen, Bert},
      author={Sarti, Alessandra},
       title={Nikulin involutions on {$K3$} surfaces},
        date={2007},
        ISSN={0025-5874},
     journal={Math. Z.},
      volume={255},
      number={4},
       pages={731\ndash 753},
         url={https://mathscinet.ams.org/mathscinet-getitem?mr=2274533},
      review={\MR{2274533}},
}

\bib{MR990136}{article}{
      author={Horozov, Emil},
      author={van Moerbeke, Pierre},
       title={The full geometry of {K}owalewski's top and {$(1,2)$}-abelian
  surfaces},
        date={1989},
        ISSN={0010-3640},
     journal={Comm. Pure Appl. Math.},
      volume={42},
      number={4},
       pages={357\ndash 407},
         url={https://mathscinet.ams.org/mathscinet-getitem?mr=990136},
      review={\MR{990136}},
}

\bib{MR0141643}{article}{
      author={Igusa, Jun-Ichi},
       title={On {S}iegel modular forms of genus two},
        date={1962},
        ISSN={0002-9327},
     journal={Amer. J. Math.},
      volume={84},
       pages={175\ndash 200},
      review={\MR{0141643}},
}

\bib{MR0184257}{article}{
      author={Kodaira, K.},
       title={On compact analytic surfaces. {II}, {III}},
        date={1963},
        ISSN={0003-486X},
     journal={Ann. of Math. (2) 77 (1963), 563--626; ibid.},
      volume={78},
       pages={1\ndash 40},
      review={\MR{0184257 (32 \#1730)}},
}

\bib{MR1554772}{article}{
      author={Kowalevski, Sophie},
       title={Sur le probl\`eme de la rotation d'un corps solide autour d'un
  point fixe},
        date={1889},
        ISSN={0001-5962},
     journal={Acta Math.},
      volume={12},
      number={1},
       pages={177\ndash 232},
         url={https://mathscinet.ams.org/mathscinet-getitem?mr=1554772},
      review={\MR{1554772}},
}

\bib{MR3263663}{article}{
      author={Kumar, Abhinav},
       title={Elliptic fibrations on a generic {J}acobian {K}ummer surface},
        date={2014},
        ISSN={1056-3911},
     journal={J. Algebraic Geom.},
      volume={23},
      number={4},
       pages={599\ndash 667},
         url={https://mathscinet.ams.org/mathscinet-getitem?mr=3263663},
      review={\MR{3263663}},
}

\bib{MR3712162}{article}{
      author={Malmendier, A.},
      author={Shaska, T.},
       title={The {S}atake sextic in {F}-theory},
        date={2017},
        ISSN={0393-0440},
     journal={J. Geom. Phys.},
      volume={120},
       pages={290\ndash 305},
         url={https://doi.org/10.1016/j.geomphys.2017.06.010},
      review={\MR{3712162}},
}

\bib{MR3731039}{article}{
      author={Malmendier, Andreas},
      author={Shaska, Tony},
       title={A universal genus-two curve from {S}iegel modular forms},
        date={2017},
        ISSN={1815-0659},
     journal={SIGMA Symmetry Integrability Geom. Methods Appl.},
      volume={13},
       pages={Paper No. 089, 17},
         url={https://mathscinet.ams.org/mathscinet-getitem?mr=3731039},
      review={\MR{3731039}},
}

\bib{MR422289}{article}{
      author={Masiewicki, Leon},
       title={Universal properties of {P}rym varieties with an application to
  algebraic curves of genus five},
        date={1976},
        ISSN={0002-9947},
     journal={Trans. Amer. Math. Soc.},
      volume={222},
       pages={221\ndash 240},
         url={https://mathscinet.ams.org/mathscinet-getitem?mr=422289},
      review={\MR{422289}},
}

\bib{MR2306633}{article}{
      author={Mehran, Afsaneh},
       title={Double covers of {K}ummer surfaces},
        date={2007},
        ISSN={0025-2611},
     journal={Manuscripta Math.},
      volume={123},
      number={2},
       pages={205\ndash 235},
         url={http://dx.doi.org/10.1007/s00229-007-0092-4},
      review={\MR{2306633}},
}

\bib{MR2804549}{article}{
      author={Mehran, Afsaneh},
       title={Kummer surfaces associated to {$(1,2)$}-polarized abelian
  surfaces},
        date={2011},
        ISSN={0027-7630},
     journal={Nagoya Math. J.},
      volume={202},
       pages={127\ndash 143},
         url={http://projecteuclid.org/euclid.nmj/1306851593},
      review={\MR{2804549}},
}

\bib{MR4063320}{article}{
      author={Milio, Enea},
       title={Computing isogenies between {J}acobians of curves of genus 2 and
  3},
        date={2020},
        ISSN={0025-5718},
     journal={Math. Comp.},
      volume={89},
      number={323},
       pages={1331\ndash 1364},
         url={https://mathscinet.ams.org/mathscinet-getitem?mr=4063320},
      review={\MR{4063320}},
}

\bib{MR0379510}{incollection}{
      author={Mumford, David},
       title={Prym varieties. {I}},
        date={1974},
   booktitle={Contributions to analysis (a collection of papers dedicated to
  {L}ipman {B}ers)},
       pages={325\ndash 350},
         url={https://mathscinet.ams.org/mathscinet-getitem?mr=0379510},
      review={\MR{0379510}},
}

\bib{MR2514037}{book}{
      author={Mumford, David},
       title={Abelian varieties},
      series={Tata Institute of Fundamental Research Studies in Mathematics},
   publisher={Published for the Tata Institute of Fundamental Research, Bombay;
  by Hindustan Book Agency, New Delhi},
        date={2008},
      volume={5},
        ISBN={978-81-85931-86-9; 81-85931-86-0},
        note={With appendices by C. P. Ramanujam and Yuri Manin, Corrected
  reprint of the second (1974) edition},
      review={\MR{2514037}},
}

\bib{MR0369362}{article}{
      author={Namikawa, Yukihiko},
      author={Ueno, Kenji},
       title={The complete classification of fibres in pencils of curves of
  genus two},
        date={1973},
        ISSN={0025-2611},
     journal={Manuscripta Math.},
      volume={9},
       pages={143\ndash 186},
         url={https://mathscinet.ams.org/mathscinet-getitem?mr=0369362},
      review={\MR{0369362}},
}

\bib{MR0429917}{article}{
      author={Nikulin, V.~V.},
       title={Kummer surfaces},
        date={1975},
        ISSN={0373-2436},
     journal={Izv. Akad. Nauk SSSR Ser. Mat.},
      volume={39},
      number={2},
       pages={278\ndash 293, 471},
      review={\MR{0429917}},
}

\bib{MR2367218}{article}{
      author={Previato, E.},
      author={Shaska, T.},
      author={Wijesiri, G.~S.},
       title={Thetanulls of cyclic curves of small genus},
        date={2007},
        ISSN={1930-1235},
     journal={Albanian J. Math.},
      volume={1},
      number={4},
       pages={253\ndash 270},
      review={\MR{2367218}},
}

\bib{MR1578134}{article}{
      author={Richelot, Fried.~Jul.},
       title={De transformatione integralium {A}belianorum primi ordinis
  commentatio},
        date={1837},
        ISSN={0075-4102},
     journal={J. Reine Angew. Math.},
      volume={16},
       pages={221\ndash 284},
         url={https://mathscinet.ams.org/mathscinet-getitem?mr=1578134},
      review={\MR{1578134}},
}

\bib{MR1578135}{article}{
      author={Richelot, Fried.~Jul.},
       title={De transformatione integralium {A}belianorum primi ordinis
  commentatio. {C}aput secundum. {D}e computatione integralium {A}belianorum
  primi ordinis},
        date={1837},
        ISSN={0075-4102},
     journal={J. Reine Angew. Math.},
      volume={16},
       pages={285\ndash 341},
         url={https://doi.org/10.1515/crll.1837.16.285},
      review={\MR{1578135}},
}

\bib{MR3781951}{article}{
      author={Ritzenthaler, Christophe},
      author={Romagny, Matthieu},
       title={On the {P}rym variety of genus 3 covers of genus 1 curves},
        date={2018},
     journal={\'{E}pijournal Geom. Alg\'{e}brique},
      volume={2},
       pages={Art. 2, 8},
         url={https://mathscinet.ams.org/mathscinet-getitem?mr=3781951},
      review={\MR{3781951}},
}

\bib{MR2296439}{incollection}{
      author={Shioda, Tetsuji},
       title={Classical {K}ummer surfaces and {M}ordell-{W}eil lattices},
        date={2007},
   booktitle={Algebraic geometry},
      series={Contemp. Math.},
      volume={422},
   publisher={Amer. Math. Soc., Providence, RI},
       pages={213\ndash 221},
         url={https://doi.org/10.1090/conm/422/08062},
      review={\MR{2296439}},
}


\bib{MR875339}{article}{
      author={Verra, Alessandro},
       title={The fibre of the {P}rym map in genus three},
        date={1987},
        ISSN={0025-5831},
     journal={Math. Ann.},
      volume={276},
      number={3},
       pages={433\ndash 448},
         url={https://mathscinet.ams.org/mathscinet-getitem?mr=875339},
      review={\MR{875339}},
}

\end{biblist}
\end{bibdiv}
\end{document}